\documentclass{amsart}

\usepackage{eufrak}

\usepackage{amssymb,amsmath, amscd, mathrsfs, yfonts, bbm, graphics}
\title[On multiplicative conditionally free convolution]
{Multiplicative c-free convolution}

\author{Mihai Popa and Jiun-Chau Wang}
\thanks{The first author was partially supported  by the grant 2-CEx06-11-34 of
the Romanian Government}
\address{Indiana University at Bloomington,
 Department of Mathematics, Rawles Hall,
 831 E 3rd St, Bloomington, IN 47405}
\email{mipopa@indiana.edu}
\address{Indiana University at Bloomington,
 Department of Mathematics, Rawles Hall,
 831 E 3rd St, Bloomington, IN 47405}
\email{jiuwang@indiana.edu}
\DeclareMathAlphabet{\mathpzc}{OT1}{pzc}{m}{it}
\newtheorem{claim}{}[section]
\newtheorem{defn}[claim]{Definition}
\newtheorem{thm}[claim]{Theorem}
\newtheorem{lemma}[claim]{Lemma}
\newtheorem{remark}[claim]{Remark}
\newtheorem{prop}[claim]{Proposition}
\newtheorem{cor}[claim]{Corollary}
\newtheorem{conseq}[claim]{Consequence}

\newcommand{\R}{{}^cR}

\newcommand{\cR}{{}^cR}
\newcommand{\cT}{{}^cT}

\newcommand{\bcnv}{\mathop{{\framebox(4,6){$\star$}}}}
\newcommand{\pbcnv}{\mathop{{\bcnv}}^\vee}

\newcommand{\gA}{\mathfrak{A}}

\newcommand{\lra}{\longrightarrow}

\newcommand{\mK}{\mathcal{K}}

\newcommand{\T}{{}^cT}

\usepackage{amssymb}
\input xy
\xyoption{all}

\begin{document}

\def\Utimes{\cup\kern-.50em\lower-.4ex\hbox{$_\times$}}
\def\utimes{\cup\kern-.86em\lower-.7ex\hbox{$_\times$}\,}
\def\Utime{\cup\!\!\!\!\lower-.6ex\hbox{$_\times$}\,}

 \maketitle
\bibliographystyle{alpha}
\begin{abstract} Using the combinatorics of non-crossing partitions,
we construct a conditionally free analogue of the Voiculescu's
$S$-transform. The result is applied to analytical description of
conditionally free multiplicative convolution and characterization
of infinite divisibility.


  $S$-transform, infinite divisibility, non-crossing and
  non-crossing linked partitions
\end{abstract}
\section{Introduction}

 The paper presents some results in conditionally free (or, shorter,
 following \cite{bls}, c-free) probability. The concept of
 c-freeness was developed in early '90's (see \cite{boca},
 \cite{bs},\cite{bls}) as a particular situation of freeness.
 Namely, if $\gA$ is an algebra and $\varphi,
 \psi:\gA\lra\mathbb{C}$ are two normalized linear functionals, then
 the family $\{\gA_j\}_j$ of subalgebras of $\gA$ is said to be
 c-free if
 \begin{enumerate}
 \item[(i)]$\psi(a_1\cdots a_n)=0$
 \item[(ii)]$\varphi(a_1\cdots a_n)=\varphi(a_1)\cdots\varphi(a_n)$
 \end{enumerate}
 for all
  $a_j\in\gA_{\varepsilon(j)}$ and
  $\varepsilon(j)\in\{1,2\},$
  such that
  $\varepsilon(1)\neq\varepsilon(2)\neq\dots\neq\varepsilon(n)$
  and
  $\psi(a_j)=0$, $1\leq j\leq n$.

 Two important tools in free probability theory are the $R$ and $S$
 transforms. Those are power series with the property that if $X$
 and $Y$ are free random variables, then $R_{X+Y}=R_X+R_Y$ and
 $S_{XY}=S_XS_Y$. While a c-free version of the $R$-transform is
 constructed in \cite{bls} and used in several papers
  (such as \cite{additivecfree}, \cite{mvp}, \cite{krystek})
   for the study of c-free additive convolution of
 measures or of c-freeness with amalgamation, the literature lacked a
 similar development for the multiplicative case.

 The present paper shows the construction of a suitable c-free
 version of the $S$-transform (in fact, as in \cite{dykema}, a more
 natural choice is its inverse, called the $T$-transform), and
 demonstrates some of its applications in limit theorems and
 characterization of infinite divisibility.

 The material is structured in seven sections. The second section
 contains notations and preliminary results used throughout the
 paper, but mostly in the third section. The third section presents
 the construction of the
 ${}^cT$-transform and proves its multiplicative property. The
 argument is based on enumerative combinatorial techniques, in the
 spirit of \cite{ns} - the lack of a Fock space model makes
 difficult an analytical procedure, as in \cite{haagerup}. The forth
 section defines the multiplicative c-free convolution of two
 measures on the unit circle and presents its connections to free
 and boolean multiplicative convolutions.
 The fifth and the sixth sections present applications of the
 multiplicative property of the ${}^cT$-transform for the study of
 limit distributions, respectively for the characterization of
 infinite divisibility. These sections are using more analytical
 techniques, analogue to \cite{berpata}, \cite{LevyHincin} and
 \cite{Mullimit}. The last section - the appendix - is showing an
 explicit combinatorial formula for computing the coefficients of
 the $T$- and $\cT$-transforms. The computations, generalizing a result from
 \cite{dykema}, are demonstrating the use of non-crossing linked
 partitions in free probability.


\section{preliminaries}

\subsection{The $R$, $T$ and $\cR$ transforms}

Let $\gA$ be a complex unital algebra endowed with two linear
functionals, $\psi$ and $\varphi$. For $a_{1},\dots,a_{n}\in\gA$,
the free cumulant $R^{n}(a_{1},\dots,a_{n})$, respectively the
c-free cumulant $\cR^{n}(a_{1},\dots,a_{n})$, are defined by the
recurrences: \begin{eqnarray*}
\psi(a_{1}\cdots a_{n}) & = & \sum_{p=1}^{n}\sum R^{p}(a_{i(1)}\cdots a_{i(p)})\left[\prod_{k=1}^{p-1}\psi(a_{i(k)+1}\cdots a_{i(k+1)-1})\right]\psi(a_{i(p)+1}\cdots a_{n})\\
\varphi(a_{1}\cdots a_{n}) & = &
\sum_{p=1}^{n}\sum\cR^{p}(a_{i(1)}\cdots
a_{i(p)})\left[\prod_{k=1}^{p-1}\psi(a_{i(k)+1}\cdots
a_{i(k+1)-1})\right]\varphi(a_{i(p)+1}\cdots a_{n})\end{eqnarray*}
 where in both lines the second summation is done over all $1=i(0)<i(1)<\dots<i(k)\leq n$.

For $X\in\gA$, we will write $R_{X}^{n}$ and $\cR_{X}^{n}$ for
$R^{n}(X,\dots,X)$, respectively $\cR^{n}(X,\dots,X)$. The $R$-,
respectively $\cR$-transform of $X$ are the formal power series
\[
{}\hspace{-1cm}R_{X}(z)=\sum_{n=1}^{\infty}r_{X}^{n}z^{n},\
\hspace{1cm}\cR_{X}(z)=\sum_{n=1}^{\infty}\cR_{X}^{n}z^{n}\]

Let $m_{X}(z)$, respectively $M_{X}(z)$ be the moment-generating
series of $X$ with respect to $\psi$ and $\varphi$, i. e.
${\displaystyle {m_{X}(z)=\sum_{n=1}^{n}\psi(X^{n})z^{n}}}$ and
${\displaystyle {M_{X}(z)=\sum_{n=1}^{n}\varphi(X^{n})z^{n}}}$. As
shown in \cite{ns} and \cite{bls}, the definitions of $R^{n}$ and
$\R^{n}$ give \begin{equation}
R\left(z[1+m_{X}(z)]\right)=m_{X}(z)\label{defRx}\end{equation}
 \begin{equation}
\R\left(z[1+m_{X}(z)]\right)\left(1+M_{X}(z)\right)=M_{X}(z)(1+m_{X}(z))\label{defcRx}\end{equation}

Two elements, $X$ and $Y$, from $\gA$ are said to be conditionally
free ( c-free) in $(\gA,\varphi,\psi)$ if the subalgebras generated
by them in $\gA$ are conditionally free, as defined in Introduction.
The key properties of the $R$ and $\cR$ transforms are summarized in
the following result:

\begin{lemma}\label{cumulantvanish} Let $X,Y$ be c-free in $\cR^{n}(X,\dots,X)$.

\begin{enumerate}
\item [(i)] Let $a_{k}\in\{ X,Y\}$, $1\leq k\leq n$. If there exist
$k\neq l$ such that $a_{k}\neq a_{l}$, then
$R^{n}(a_{1},\dots,a_{n})=0$ and $\cR^{n}(a_{1},\dots,a_{n})=0$
\item [(ii)] Let $T_{X}$ be the formal power series defined by ${\displaystyle {T_{X}(z)=\left(\frac{1}{z}R(z)\right)\circ\left(R^{\langle-1\rangle}(z)\right)}}$,
where $(F(z))^{\langle-1\rangle}$ is the substitutional inverse of
the formal series $F(z)$. Then \[ T_{XY}(z)=T_{X}(z)T_{Y}(z).\]

\end{enumerate}
\end{lemma}

Proofs for the first part of (i) can be found, for example, in
\cite{haagerup}, Theorem 2.2, and \cite{ns}, Theorem 16.16. Second
part of (i) is proved in \cite{bls}, Theorem 3.1.

For (ii), in \cite{haagerup}, Corollary 2.4. or Theorem 3.5, and
\cite{ns}, Corollary 18.17, the property is proved for
$S(z)=\left[T(z)\right]^{-1}$ (inverse with respect to
multiplication). As in \cite{dykema}, for the purpose of this paper
it will be simpler to consider $T(z)=\left[S(z)\right]^{-1}$.

\subsection{Boolean independence, the $\eta$- and $B$-transforms}

Let now $\gA$ be a unital algebra and $\varphi:\gA\lra\mathbb{C}$ be
a normalized functional. Two subalgebras $\gA_{1},\gA_{2}$ are said
to be \emph{boolean independent} if \[
\varphi(x_{1}y_{1}x_{2}y_{2}\cdots)=\varphi(x_{1})\varphi(y_{1})\cdots,\]
 for all $x_{k}\in\gA_{1}$ and $y_{k}\in\gA_{2}$. We will say that
$X,Y\in\gA$ are boolean independent if the nonunital algebras
generated in $\gA$ by $X$, respectively $Y$ are boolean independent.

If $\gA=\mathbb{C}\oplus\gA_{0}$ (direct sum of vector spaces), and
$\psi=Id_{\mathbb{C}}\oplus0_{\gA_{0}}$, then conditional freeness
with respect to $(\varphi,\psi)$ is equivalent to boolean
independence with respect to $\varphi$.

For the results in Section 4-6, we need the definitions and results
below (see \cite{booleanmvp} or \cite{ufranz} for proves).

\begin{defn}
\label{defboolean} Let $X\in\gA$ and $M_{X}(z)$ be the moment
generating series of $X$ with respect to $\varphi$, that is
${\displaystyle {M_{X}(z)=\sum_{n=1}^{\infty}\varphi(X^{n})z^{n}}}$.
The $\eta$-, respectively $B$-transforms of $X$ are the formal power
series given by the relations
$\eta_{X}(z)=\frac{M_{X}(z)}{1+M_{X}(z)}$, respectively
$B_{X}(z)=\frac{1}{z}\eta_{X}(z)$.
\end{defn}
\begin{prop}
If $X$ and $Y$ are two boolean independent elements of $\gA$, then:
\begin{enumerate}
\item [(i)]$\eta_{X+Y}(z)=\eta_{X}(z)+\eta_{Y}(z).$
\item [(ii)]$B_{X+Y}(z)=B_{X}(z)+B_{Y}(z)$\ and\
$B_{(1+X)(1+Y)}(z)=B_{(1+X)}(z)\cdot B_{(1+Y)}(z).$
\end{enumerate}
\end{prop}

\subsection{Non-crossing partitions}

By a partition on the ordered set $\langle n\rangle=\{1,2,\dots,n\}$
we will understand a collection of mutually disjoint subsets of
$\langle n\rangle$, $\gamma=(B_{1},\dots,B_{q})$, called
\emph{blocks} whose union is the entire set $\langle n\rangle$. A
\emph{crossing} is a sequence $i<j<k<l$ from $\langle n\rangle$ with
the property that there exist two different blocks $B_{r}$ and
$B_{s}$ such that $i,k\in B_{r}$ and $j,l\in B_{s}$. A partition
that has no crossings will be called non-crossing. The set of all
non-crossing partitions on $\langle n\rangle$ will be denoted by
$NC(n)$.

For $\gamma\in NC(n)$, a block $B=(i_{1},\dots,i_{k})$ of $\gamma$
will be called \emph{interior} if there exists another block
$D\in\gamma$ and $i,j\in D$ such that $i<i_{1},i_{2},\dots,i_{k}<j$.
A block will be called \emph{exterior} if is not interior. The set
of all interior, respectively exterior blocks of $\gamma$ will be
denoted by $Int(\gamma)$, respectively $Ext(\gamma)$.

With the above notations, the recurrences defining $R^{n}$,
respectively $\R^{n}$ can be written as
  \begin{eqnarray*}
\psi(a_{1}\cdots a_{n})
 & = &
 \sum_{\gamma\in NC(n)}\prod_{\substack{B\in\gamma\\
B=(i_{1},\dots,i_{k})}
}
 R^{k}(a_{i_{1}},\dots,a_{i_{k}})\\
\varphi(a_{1}\cdots a_{n})
 & = &
 \sum_{\gamma\in NC(n)}
  \prod_{
  \substack{D\in Ext(\gamma)\\
D=(i_{1},\dots,i_{k})}
 }
 \R^{k}(a_{i_{1}},\dots,a_{i_{k}})\prod_{\substack{B\in Int(\gamma)\\
B=(j_{1},\dots,j_{s})}
 }
 R^{s}(a_{j_{1}},\dots,a_{j_{s}})
 \end{eqnarray*}

$NC(n)$ has a lattice structure with respect to the reversed
refinement order, with the biggest, respectively smallest element
$\mathbbm{1}_{n}=(1,2,\dots,n)$, respectively $0_{n}=(1),\dots,(n)$.
For $\pi,\sigma\in NC(n)$ we will denote by $\pi\bigvee\sigma$ their
join (smallest common upper bound).

The Kreweras complementary
  $Kr(\pi)$ of $\pi\in NC(n)$
  is defined as follows. Consider the symbols
   $\overline{1},\dots,\overline{n}$
   such that
   $1<\overline{1}<2<\dots<n<\overline{n}$.
   Then $Kr(\pi)$ is the
biggest element of
 $NC(\overline{1},\dots,\overline{n})\cong NC(n)$
such that
  \[
  \pi\cup Kr(\pi)\in
NC(1,\overline{1},\dots,n,\overline{n}).
 \]

The following notations will also be used:
 \begin{eqnarray*}
NC_{1}(n) & = & \{\sigma:\sigma\in NC(n),\sigma\
\text{has only one exterior block}\}\\
NC_{2}(n) & = & \{\sigma:\sigma\in NC(n),\sigma\
\text{has only two exterior blocks}\}\\
NC_{S}(n) & = & \{\sigma:\sigma\in NC(2n),\sigma\
\text{the elements from the same block of $\sigma$ have}\\
 &  & \text{the same parity}\}.
 \end{eqnarray*}
 For $\sigma\in NC_{S}(2n)$, denote $\sigma_{+}$, respectively $\sigma_{-}$
the restriction of $\sigma$ to the even, respectively odd, numbers
from $\{1,2,\dots,2n\}$. Define
 \[
  NC_{0}(n)=\{\sigma:\sigma\in
NC(n),\sigma_{+}=Kr(\sigma_{-})\}.
 \]
 Also, we will need to consider the mappings
  \[
NC(n)\ni\sigma\mapsto\widehat{\sigma}\in NC(2n)
 \]
 constructing by doubling the elements, and
   \[
NC(n)\times NC(m)\ni(\pi,\sigma)\mapsto\pi\oplus\sigma\in NC(m+n),
 \]
 the juxtaposition of partitions.


\section{the $\T$-transform}

Let $\gA$ be a unital algebra,
 $\varphi,\psi:\gA\lra\mathbb{C}$
 be two normalized linear functionals and $X$ be n element from $\gA$.
Denote by $m_{X}(z)$, respectively $M_{X}(z)$, the moment generating
power series of $X$ with respect to $\psi$ and $\varphi$, as in the
formulas (\ref{defRx}), (\ref{defcRx}).

In this section we will construct a formal power series
${}^{c}T_{X}(z)$ such that:

\begin{enumerate}
\item [(i)]for $\varphi=\psi$ one has that
 $T_{X}(z)={}^{c}T_{X}(z)$
(see \ref{cumulantvanish} for the definition of $T_{X}(z)$).
\item [(ii)] $M_{X}(z)$ can be obtained from $T_{X}(z)$ and ${}^{c}T_{X}(z)$
via substitutional composition, substitutional inverse and algebraic
operations.
\item [(iii)]if $X$ and $Y$ are c-free elements from $\gA$, then
 \[
{}^{c}T_{XY}(z)={}^{c}T_{X}(z){}^{c}T_{Y}(z).
 \]
\end{enumerate}
\noindent The construction of ${}^{c}T_{X}(z)$ is presented as a
natural consequence of the combinatorial properties of the free and
c-free cumulants $R^{n}$ and $\R^{n}$.

For $a_{1},\dots,a_{n}\in\gA$ we need to consider the multilinear
maps
  \begin{eqnarray*}
\kappa_{\pi}[a_{1},\dots,a_{n}]
 & = &
 \prod_{\substack{B\in\pi\\
B=(i(1),\dots,i(p))\\
\\}
 }
 R^{p}(a_{i(1)}\cdots a_{i(p)})\\
\mK_{\pi}[a_{1},\dots,a_{n}]
 & = &
  \Big[\prod_{\substack{D\in Ext(\pi)\\
D=(i(1),\dots,i(p))}
}\R^{p}(a_{i(1)}\cdots a_{i(p)})\Big]\Big[\prod_{\substack{B\in Int(\pi)\\
B=(i(1),\dots,i(q))\\
\\}
}R^{q}(a_{i(1)}\cdots a_{i(q)})\Big]
 \end{eqnarray*}

First, let us remark the following result on free and c-free
cumulants with constants among entries:

\begin{lemma}
 For all $a_{1},\dots,a_{n}\in\gA$, one has that
  \begin{eqnarray*}
R^{n+1}(a_{1},\dots,a_{j},1,a_{j+1},\dots,a_{n})
 & = & 0\\
\R^{n+1}(a_{1},\dots,a_{j},1,a_{j+1},\dots,a_{n})
 & = &
0
 \end{eqnarray*}

\end{lemma}

\begin{proof}
The first part is proved in \cite{ns}. We will prove the second part
by induction. First, since $\varphi(a)=\varphi(a1)$, the definition
of $\R^{n}$ gives
 \[
  \R^{2}(a,1)+\R^{1}(a)1=\R^{1}(a)
   \]
 therefore
  $\R^{2}(a,1)=0$.

In general,
 $\varphi(a_{1}\cdots a_{n})
  =
   \varphi(a_{1}\cdots a_{j}1a_{j+1}\cdots a_{n}).$
    But
  \[
   \varphi(a_{1}\cdots a_{n})
   =
   \sum_{\pi\in NC(n)}\mK_{\pi}[a_{1},\dots,a_{n}]
  \]
 while
  \begin{eqnarray*}
\varphi(a_{1}\cdots a_{j}1a_{j+1}\cdots a_{n})
 & = &
  \sum_{\pi\in NC(n+1)}\mK_{\pi}[a_{1},\dots,a_{j},1,a_{j+1,}\dots,a_{n}]\\
 && \hspace{-2cm}=
  \sum_{
   \substack{\pi\in NC(n+1)\\
(1)\in\pi}
 }
  \mK_{\pi}[a_{1},\dots,a_{j},1,a_{j+1,}\dots,a_{n}]
  +R^{n+1}(a_{1},\dots,1,\dots,a_{n})\\
 &  &\hspace{-1.6cm}
  +\sum_{\substack{\pi\in NC(n+1)\\
(1)\notin\pi\\
\pi\neq\mathbbm{1}_{n+1}}
 }
 \mK_{\pi}[a_{1},\dots,a_{j},1,a_{j+1,}\dots,a_{n}]
 \end{eqnarray*}

The last term cancels from the induction hypothesis. Consider the
bijection
 $\pi\mapsto\pi^{\prime}$
   between
   $\{\pi:\pi\in NC(n+1),(j)\in\pi\}$
 and
  $NC(n)$, where $\pi^{\prime}$ is obtained
by erasing the block $(j)$ from $\pi$. It has the property that
 \[
\mK_{\pi}[a_{1},\dots,a_{j},1,a_{j+1,}\dots,a_{n}]
 =\mK_{\pi^{\prime}}[a_{1},\dots,a_{n}]
 \]
 Therefore
  \[
\varphi(a_{1}\cdots a_{j}1a_{j+1}\cdots a_{n})
 =
 \varphi(a_{1}\cdots a_{n})
 +\R^{n+1}[a_{1},\dots,a_{j},1,a_{j+1,}\dots,a_{n}]
 \]
 hence the conclusion.
\end{proof}


We will focus next on free and c-free cumulants with products as
entries.

\begin{lemma}\label{decompfree} Let $X,Y\in\gA$, c-free. For all
$\pi\in NC(n)$ one has that:

\begin{enumerate}
\item [(a)] $\kappa_{\pi}[XY,\dots,XY]=\sum_{\substack{\sigma\in NC_{S}(2n)\\
\sigma\bigvee\widehat{0_{n}}=\widehat{\pi}}
}\kappa_{\sigma}[X,Y,\dots,X,Y].$

\noindent Particularly, \[ R_{XY}^{n}=\sum_{\sigma\in
NC_{0}(2n)}\kappa_{\sigma}[X,Y,\dots,X,Y].\]

\item [(b)] $\mK_{\pi}[XY,\dots,XY]
 =
 \sum_{\substack{\sigma\in NC_{S}(2n)\\
\sigma\bigvee\widehat{0_{n}}
 =\widehat{\pi}}
}\mK_{\sigma}[X,Y,\dots,X,Y].$

\noindent Particularly, \[ \R_{XY}^{n}=\sum_{\sigma\in
NC_{0}(2n)}\mK_{\sigma}[X,Y,\dots,X,Y].\]

\end{enumerate}
\end{lemma}

\begin{proof}
(a) is shown in \cite{ns}. The proof of (b) will be done by
induction on $n$. For $n=1$, the statement is trivial:
  \[
\R^{1}[XY]=\varphi(XY)=\varphi(X\cdot
Y)=\R^{2}(X,Y)+\R^{1}(X)\R^{1}(Y)=\R^{1}(X)\R^{1}(Y)
 \]
 sice the mixed c-free cumulants vanish (see \ref{cumulantvanish}).

For the induction step, we distinguish two cases:

Case 1: $\pi\neq\mathbbm{1}_{n}$.

If $\pi$ has more than one exterior block, then
$\pi=\pi_{1}\oplus\pi_{2}$ for some $m<n$, $\pi_{1}\in
NC(m),\pi_{2}\in NC(n-m)$. One has that
 \[
\mK_{\pi}[XY,\dots,XY]=\mK_{\pi_{1}}[XY,\dots,XY]\mK_{\pi_{2}}[XY,\dots,XY],
 \]
 and the result follows from the induction hypothesis.

If $\pi$ has exactly one exterior block, then
 \[
\mK_{\pi}[XY,\dots,XY]=\R^{p}[Xy,\dots,XY]\kappa_{\overline{\pi}}[XY,\dots,XY],
 \]
 where $p$ is the length of the exterior block of $\pi$ and $\overline{\pi}$
is the non-crossing partition obtained by erasing the exterior block
of $\pi$. The result follows from \ref{decompfree} and the induction
hypothesis.

Case 2: $\pi=\mathbbm{1}_{n}$.

We need to show that
\begin{equation}
\label{2}
 \R_{XY}^{n}
 =
 \sum_{\sigma\in NC_{0}(2n)}\mK_{\sigma}[X,Y,\dots,X,Y].
 \end{equation}

Before proceeding with the proof, let us take a better look at the
right-hand side of \ref{2}. Since any $\sigma\in NC_{0}(2n)$ has
exactly 2 exterior blocks, one containing 1 and one containing $2n$,
each
 $\mK_{\sigma}[X,Y,\dots,X,Y]$
  will have exactly two factors of
the type $\R^{m}$, namely $\R_{X}^{p}$ and $\R_{Y}^{q}$, where $p,q$
are, respectively, the length of the first and second exterior block
of $\sigma$. Also, $\sigma\in NC_{0}(2n)$ implies that all other
factors of
 $\mK_{\sigma}[X,Y,\dots,X,Y]$
  are of the form $R_{X}^{s}$
or $R_{Y}^{t}$.

One has that
 \begin{eqnarray*}
\varphi\left((XY)^{n}\right) & = & \varphi(XY\cdot XY\cdots XY)\\
 & = &
  \sum_{\pi\in NC(n)}\mK_{\pi}[XY,\dots,XY]\\
 & = &
  \R_{XY}^{n}+\sum_{\substack{\pi\in NC(n)\\
\pi\neq\mathbbm{1}_{n}} }\mK_{\pi}[XY,\dots,XY]
 \end{eqnarray*}

On the other hand,
 \begin{equation}
\varphi\left((XY)^{n}\right)=\varphi(X\cdot Y\cdots X\cdot
Y)=\sum_{\sigma\in
NC(2n)}\mK_{\sigma}[X,Y,\dots,X,Y]\label{4}
 \end{equation}
 Since all the mixed cumulants vanish, \ref{4} becomes:
  \begin{eqnarray*}
\varphi\left((XY)^{n}\right) & = & \sum_{\sigma\in NC_{S}(2n)}\mK_{\sigma}[X,Y,\dots,X,Y]\\
 & = & \sum_{\sigma\in NC_{0}(2n)}\mK_{\sigma}[X,Y,\dots,X,Y]+\sum_{\substack{\sigma\in NC_{S}(2n)\\
\sigma\notin NC_{0}(2n)} }\mK_{\sigma}[X,Y,\dots,X,Y]
  \end{eqnarray*}

But

 $NC_{S}(2n)=\bigcup_{\pi\in NC(n)}\{\sigma:\ \sigma\in
NC_{S}(2n),\sigma\bigvee\widehat{0_{n}}=\widehat{\pi}\}$.
  Also, for
$\sigma\in NC_{s}(2n)$, one has that $\sigma\in NC_{0}(2n)$ if and
only if
 $\sigma\bigvee\widehat{0_{n}}=\mathbbm{1}_{2n}$.
  Therefore:
 \[
NC_{S}(2n)\setminus NC_{0}(2n)=\bigcup_{\substack{\pi\in NC(n)\\
\pi\neq\mathbbm{1}_{n}} }\{\sigma:\ \sigma\in
NC_{S}(2n),\sigma\bigvee\widehat{0_{n}}=\widehat{\pi}\}.
 \]

It follows that
 \begin{eqnarray*}
\sum_{\substack{\sigma\in NC_{S}(2n)\\
\sigma\notin NC_{0}(2n)}
}\mK_{\sigma}[X,Y,\dots,X,Y]
 & = &
  \sum_{\substack{\pi\in NC(n)\\
\pi\neq\mathbbm{1}_{n}}
}\sum_{\substack{\sigma\in NC_{S}(2n)\\
\sigma\bigvee\widehat{0_{n}}=\widehat{\pi}}
}\mK_{\sigma}[X,Y,\dots,X,Y]\\
 & = &
  \sum_{\substack{\pi\in NC(n)\\
\pi\neq\mathbbm{1}_{n}} }\mK[XY,\dots,XY].
 \end{eqnarray*}
 so the proof is now complete.
\end{proof}
 For stating the next result we need first a brief review of the
operation $\bcnv$ (boxed convolution), as defined in \cite{ns}.

Consider the formal power series ${\displaystyle
{f(z)=\sum_{n=1}^{\infty}\alpha_{n}z^{n}}}$. For $\pi\in NC(n)$, we
define \[ Cf_{\pi}(f)=\prod_{B\in\pi}\alpha_{|B|}.\]

If ${\displaystyle {g(z)=\sum_{n=1}^{\infty}\beta_{n}z^{n}}}$ is
another formal power series, we define their boxed convolution
\setlength{\unitlength}{.035cm}
\[
f(z)\bcnv g(z)=\left(f\bcnv
g\right)(z)=\sum_{n=1}^{\infty}\gamma_{n}z^{n}\]
 by: \[
\gamma_{n}=\sum_{n\in NC(n)}Cf_{\pi}(f)\cdot Cf_{Kr(\pi)}(g).\]

We will need the following two results proved in \cite{ns}:

\begin{lemma}\label{lem2} Suppose $\alpha_{1}\neq0$. Then
\[
f^{\langle-1\rangle}\circ\left(f\bcnv
g\right)=\frac{1}{\alpha_{1}}\left(f\pbcnv g\right)
 \]
 where
 \[
\left(f\pbcnv g\right)(z)=\sum_{n=1}^{\infty}\lambda_{n}z^{n}\]
 for \[
\lambda_{n}=\sum_{\substack{\pi\in NC(n)\\
(1)\in\pi} }Cf_{\pi}(f)\cdot Cf_{Kr(\pi)}(g)
 \]
 \end{lemma}

\begin{lemma}If $X$ and $Y$ are free, then
\[
R_{XY}(z)=R_{X}(z)\bcnv R_{Y}(z).
 \]
 \end{lemma}

Now we can proceed with the main theorems of this section. To ease
the notations, say that
 ${\displaystyle
{R_{X}(z)=\sum_{n=1}^{\infty}\alpha_{n}z^{n}}}$
 and
  ${\displaystyle
{R_{Y}(z)=\sum_{n=1}^{\infty}\beta_{n}z^{n}}}$.

\begin{thm}
\label{th1} If $X,Y$ are c-free, then
 \[
\frac{1}{z}\R_{XY}(z)=\left[\left(\frac{1}{z}\R_{X}\right)\circ\left(\frac{1}{\alpha_{1}}R_{X}\pbcnv
R_{Y}\right)\right]\left[\left(\frac{1}{z}\R_{Y}\right)\circ\left(\frac{1}{\beta_{1}}R_{Y}\pbcnv
R_{X}\right)\right]
 \]

\end{thm}
\begin{proof}
As shown before,

\begin{equation}
\R_{XY}^{n}=\sum_{\sigma\in
NC_{0}(2n)}\mK_{\sigma}[X,Y,\dots,X,Y].\label{5}\end{equation}

Each $\sigma\in NC_{0}(2n)$ has exactly 2 exterior blocks, one
consisting on
 $1=b_{1}<b_{2}<\dots<b_{p}$
  and the other one consisting on
  $b_{p}+1=d_{1}<d_{2}<\dots<d_{q}=2n$.
  Let us denote
$\pi_{k}$ the restriction of $\sigma$ to
$(b_{k}+1,b_{k}+2,\dots,b_{k+1}-1)$
 and $\omega_{l}$ the restriction
of $\sigma$ to $(d_{l}+1,d_{l}+2,\dots,d_{l+1}-1)$. Then:
  \[
\mK_{\sigma}[X,Y,\dots,X,Y]
 =
 \R_{X}^{p}\cdot
  \left(
  \prod_{k=1}^{p-1}\kappa_{\pi_{k}}[Y,X,\dots,X,Y]
  \right)
  \cdot\R_{Y}^{q}\cdot
   \left(
    \prod_{l=1}^{q-1}\kappa_{\omega_{l}}[X,Y,\dots,X,Y,X]
    \right)
     \]
 Sice $\sigma\in NC_{0}(2n)$, one has that
  $(1)\oplus\pi_{k}\in NC_{0}(b_{k+1}-b_{k}+1)$
and
 $((1)\oplus\omega_{l}\in NC_{0}(d_{l+1}-d_{l}+1)$,
  therefore:
\begin{eqnarray*}
\kappa_{\pi_{k}}[Y,X,\dots,Y]
 & = &
  \frac{1}{\alpha_{1}}\kappa_{(1)\oplus\pi_{k}}[X,Y,\dots,X,Y]\\
 & = &
  \frac{1}{\alpha_{1}}Cf_{((1)\oplus\pi_{k})_{-}}(R_{X})
   \cdot Cf_{Kr(((1)\oplus\pi_{k})_{-})}(R_{Y})
   \end{eqnarray*}
 and
  \begin{eqnarray*}
\kappa_{\omega_{l}}[X,Y,X,\dots,X]
 & = &
  \frac{1}{\beta_{1}}\kappa_{(1)\oplus\omega_{l}}[X,Y,\dots,X,Y]\\
 & = &
  \frac{1}{\beta_{1}}Cf_{((1)\oplus\omega_{l})_{-}}(R_{Y})\cdot
   Cf_{Kr(((1)\oplus\omega_{l})_{-})}(R_{X})
   \end{eqnarray*}
 hence q.e.d..
\end{proof}
\begin{thm}
\label{mainrez}
  If $X,Y$ are c-free such that $psi(X)\neq 0 \neq \psi(Y)$, then
    \[
\left[
\left(\frac{1}{z}\R_{XY}\right)\circ\left(R_{XY}^{\langle-1\rangle}(z)\right)
 \right]
  =
  \left[
  \left(\frac{1}{z}\R_{X}\right)\circ\left(R_{X}^{\langle-1\rangle}(z)\right)
   \right]
  \cdot
  \left[
  \left(\frac{1}{z}\R_{Y}\right)\circ\left(R_{Y}^{\langle-1\rangle}(z)\right)
  \right]
  .\]

\end{thm}
\begin{proof}
From \ref{th1}, one has:
\[
\frac{1}{z}\R_{XY}(z)
 =
 \left[
  \left(\frac{1}{z}\R_{X}\right)
   \circ\left(\frac{1}{\alpha_{1}}R_{X}\pbcnv R_{Y}\right)
 \right]
 \left[
 \left(
  \frac{1}{z}\R_{Y}\right)\circ\left(\frac{1}{\beta_{1}}R_{Y}\pbcnv R_{X}
  \right)
  \right]
   \]
 Using \ref{lem2}, the above equality becomes
  \begin{eqnarray*}
\frac{1}{z}\R_{XY}(z)
  & = &
  \left[
   \left(\frac{1}{z}\R_{X}\right)\circ
   \left(R_{X}^{\langle-1\rangle}\circ\left(R_{X}\pbcnv R_{Y}\right)\right)
  \right]
  \left[
  \left(\frac{1}{z}\R_{Y}\right)\circ
  \left(R_{Y}^{\langle-1\rangle}\circ\left(R_{Y}\pbcnv R_{X}\right)\right)
  \right]\\
 & = &
  \left[
   \left(\left(\frac{1}{z}\R_{X}\right)\circ R_{X}^{\langle-1\rangle}\right)
   \circ\left(R_{X}\pbcnv R_{Y}\right)
  \right]
   \left[
   \left(\left(\frac{1}{z}\R_{Y}\right)\circ R_{Y}^{\langle-1\rangle}\right)
   \circ\left(R_{Y}\pbcnv R_{X}\right)
   \right]
 \end{eqnarray*}
 But
  $R_{XY}=r_{X}\bcnv R_{y}$,
  so
  \begin{equation*}
\frac{1}{z}\R_{XY}(z)
 =\left[
  \left(\left(\frac{1}{z}\R_{X}\right)\circ R_{X}^{\langle-1\rangle}\right)
   \circ\left(R_{XY}\right)\right]\left[\left(\left(\frac{1}{z}\R_{Y}\right)\circ
R_{Y}^{\langle-1\rangle}\right)\circ\left(R_{XY}\right)
 \right]
  \end{equation*}

Finally, composing at the right to $R_{XY}^{\langle-1\rangle}$ one
gets the conclusion.
\end{proof}
\begin{cor}
The power series ${}^{c}T_{x}(z)$ defined by \[
{}^{c}T_{X}(z)=\left(\frac{1}{z}\R_{X}(z)\right)\circ\left(R^{\langle-1\rangle}(z)\right)\]
 satisfies the properties (i)-(iii) described in the beginning of
the section.
\end{cor}
\begin{proof}
(iii) is proved in Theorem \ref{mainrez}. (i) is an immediate
consequence of comparing the definitions of $T_{X}$ (as in Lemma
\ref{cumulantvanish}) and of ${}^{c}T_{X}$. Finally, (ii) is implied
by (\ref{defRx}), (\ref{defcRx}) and the definition of
${}^{c}T_{X}$.
\end{proof}


\section{Multiplicative conditionally free convolution of measures}

Denote by
 $\mathcal{M}_{\mathbb{T}}$
the family of all Borel
probability measures supported on the unit circle $\mathbb{T}$ and
by
 $\mathcal{M}_{\mathbb{T}}^{\times}$
  the set of all measures
 $\nu\in\mathcal{M}_{\mathbb{T}}$
 such that
 $\int_{\mathbb{T}}\zeta\, d\nu(\zeta)\neq0$.

For $\mu\in\mathcal{M}_{\mathbb{T}}$ we denote
 \[
  {\displaystyle {m_{\mu}(z)
   =
   \int_{\mathbb{T}}\frac{z\zeta}{1-z\zeta}d\mu(\zeta)}}
  \]
 the moment generating function of $\mu$, analytic in the unit disk
$\mathbb{D}$.

Consider now
 $\mu\in\mathcal{M}_{\mathbb{T}}$
   and
$\nu\in\mathcal{M}_{\mathbb{T}}^{\times}.$
 To the pair $(\mu,\nu)$
we associate the functions
 $R_{\nu}(z)$, $\R_{\mu,\nu)}(z)$,
 $T_{\nu}(z)$, ${}^{c}T_{(\mu,\nu)}(z),$
 analytic in some
neighborhood of zero, and given by the following relations:
\begin{eqnarray*}
R_{\nu}\left(z[1+m_{\nu}(z)]\right)
 & = &
  m_{\nu}(z)\\
\R_{(\mu,\nu)}\left(z[1+m_{\nu}(z)\right)\cdot(1+m_{\mu}(z))
 & = &
   m_{\mu}(z)\cdot(1+m_{\nu}(z))\\
T_{\nu}(z)
 & = &
  \left(
   \frac{1}{z}R_{\nu}(z)\right)\circ\left(R_{\nu}^{\langle-1\rangle}(z)
    \right)\\
{}^{c}T_{(\mu,\nu)}(z)
 & = &
 \left(
 \frac{1}{z}\R_{\mu,\nu)}(z)\right)\circ\left(R_{\nu}^{\langle-1\rangle}(z)
  \right)
  \end{eqnarray*}

\begin{defn}
If
$(\mu_{1},\nu_{1}),(\mu_{2},\nu_{2})
 \in
 \mathcal{M}_{\mathbb{T}}\times\mathcal{M}_{\mathbb{T}}^{\times}$,
their multiplicative conditionally free convolution,
$(\mu_{1},\nu_{1})\boxtimes(\mu_{2},\nu_{2})$
 is the unique pair
$(\mu,\nu)\in\mathcal{M}_{\mathbb{T}}\times\mathcal{M}_{\mathbb{T}}^{\times}$
such that
$T_{\nu}(z)=T_{\nu_{1}}(z)\cdot T_{\nu_{2}}(z)$\
 and\
${}^{c}T_{(\mu,\nu)}(z)
 =
  {}^{c}T_{(\mu_{1},\nu_{1})}(z)\cdot{}^{c}T_{(\mu_{2},\nu_{2})}(z).$
\end{defn}
Particularly, for $X,Y$ c-free elements of $\gA$, if
$(\mu_{1},\nu_{1}),(\mu_{2},\nu_{2})$
are the distributions of $X$, respectively $Y$, then
 $(\mu_{1},\nu_{1})\boxtimes(\mu_{2},\nu_{2})$
is the distribution of
$XY$.

The next two sections address some analytical properties of the
operation $\boxtimes$. More precisely, Section 5 states some limit
theorems and Section 6 describes $\boxtimes$ infinite divisibility.

It will be convenient to introduce a variation of the function
${}^{c}T_{(\mu,\nu)}$ as follows. Namely, for the measures
$\mu\in\mathcal{M}_{\mathbb{T}}$
 and
$\nu\in\mathcal{M}_{\mathbb{T}}^{\times}$,
 we consider the analytic
function

\begin{equation}
\Sigma_{(\mu,\nu)}(z)
 =
  {}^{c}T_{(\mu,\nu)}\left(\frac{z}{1-z}\right)\label{defsigma}
   \end{equation}
 in a neighborhood of zero.

Similarly to Definition \ref{defboolean}, we denote
 ${\displaystyle
{\eta_{\mu}(z)=\frac{m_{\mu}(z)}{1+m_{\mu}(z)}}}$
 and
${\displaystyle {B_{\mu}(z)=\frac{m_{\mu}(z)}{z(1+m_{\mu}(z))}}}$

\begin{lemma}\label{Bnu} For ${\eta_{\nu}}(z)$ and $B_{\mu}(z)$,
defined above, one has:
\[
\Sigma_{(\mu,\nu)}
 =\left[\frac{1}{z}\eta_{\mu}(z)\right]\circ\left[\eta_{\nu}^{\langle-1\rangle}(z)\right]
  =
   B_{\mu}\left(\eta_{\nu}^{\langle-1\rangle}(z)\right)
   \]
 \end{lemma}

\begin{proof}
The relation
$R_{\nu}\left(z[1+m_{\nu}(z)]\right)=m_{\nu}(z)$
 (i. e. the
definition of $R_{\nu}$) implies that
 \[
R_{\nu}^{\langle-1\rangle}(z)=(1+z)m_{\nu}^{\langle-1\rangle}(z),
 \]
 hence
  \begin{equation}
R_{\nu}^{\langle-1\rangle}(m_{\nu}(z))=z(1+m_{\nu}(z))\label{eq11}
 \end{equation}

Let us now consider the definition of $\R_{(\mu,\nu)}$,
 \[
\R_{(\mu,\nu)}\left(z[1+m_{\nu}(z)\right)\cdot(1+m_{\mu}(z))=m_{\mu}(z)\cdot(1+m_{\nu}(z))
 \]
 which implies that
 \[
\frac{\R_{(\mu,\nu)}\left(z(1+m_{n}u(z))\right)}{z(1+m_{n}u(z))}=\frac{m_{m}u(z)}{z(1+m_{m}u(z))}.
  \]
 Taking into consideration (\ref{eq11}) it follows that
 \[
\frac{\R_{(\mu,\nu)}
 \left(R_{\nu}^{\langle-1\rangle}(m_{\nu}(z))\right)}{R_{\nu}^{\langle-1\rangle}(m_{\nu}(z))}
  =B_{\mu}(z),
  \]
 that is
  \[
{}^{c}T_{(\mu,\nu)}(m_{\nu}(z))=B_{\mu}(z)
 \]
 and composing at right with $m_{\nu}^{\langle-1\rangle}$ we get
the conclusion.
\end{proof}
Note that
 \[
\Sigma_{(\mu,\nu)}(0)=B_{\mu}(0)=\int_{\mathbb{T}}\zeta\,
d\mu(\zeta)
 \]
 and the function $\Sigma_{(\delta_{\lambda},\nu)}(z)$ is the constant
function $\lambda$.

As observed in \cite{Semigroup}, the function $B_{\mu}$ maps
$\mathbb{D}$ into $\overline{\mathbb{D}}$, and, conversely, any
analytic function
$B:\,\mathbb{D}\rightarrow\overline{\mathbb{D}}$
 is of the form
$B_{\mu}$ for a unique probability measure $\mu$ on $\mathbb{T}$. As
a consequence of this observation, the function $\Sigma_{(\mu,\nu)}$
is uniformly bounded by $1$.
 Moreover, if
$\nu\in\mathcal{M}_{\mathbb{T}}^{\times}$
 and $\Sigma$ is an
analytic function defined on the set
  $\eta_{\nu}(\mathbb{D})$ such
that the function $\Sigma$ is uniformly bounded by $1$, then there
exists a unique probability measure $\mu$ on $\mathbb{T}$ such that
$\Sigma=\Sigma_{(\mu,\nu)}$.

Before starting Section 5, we will finally mention the following
construction. If $\mu,\nu$ are two probability measures on
$\mathbb{T}$, their \emph{boolean convolution} $\mu\Utime\nu$ is the
unique measure on $\mathbb{T}$ given by
 \[
B_{\mu\Utime\nu}(z)=B_{\mu}(z)B_{\nu}(z)
 \]


\section{Limit Theorems}

Let $\{ k_{n}\}_{n=1}^{\infty}$ be a sequence of positive integers.
Consider two \emph{infinitesimal} triangular arrays $\{\mu_{nk}:\,
n\in\mathbb{N},1\leq k\leq k_{n}\}$ and $\{\nu_{nk}:\,
n\in\mathbb{N},1\leq k\leq k_{n}\}$ in
$\mathcal{M}_{\mathbb{T}}^{\times}$.
Here the infinitesimality means
that
\[
  \lim_{n\rightarrow\infty}\max_{1\leq k\leq
k_{n}}\mu_{nk}(\{\zeta\in\mathbb{T}:\,\left|\zeta-1\right|\geq\varepsilon\})=0,
 \]
 for every $\varepsilon>0$. We say that a sequence
 $\{\mu_{n}\}_{n=1}^{\infty}\subset\mathcal{M}_{\mathbb{T}}$
converges \emph{weakly} to a measure
$\mu\in\mathcal{M}_{\mathbb{T}}$
 if
 \[
\lim_{n\rightarrow\infty}\int_{\mathbb{T}}f(\zeta)\,
d\mu_{n}(\zeta)=\int_{\mathbb{T}}f(\zeta)\, d\mu(\zeta),
 \]
 for every
continuous function $f$ on $\mathbb{T}$.
The weak convergence of a sequence of pairs
$\{(\mu_{n},\nu_{n})\}_{n=1}^{\infty}$
simply means the componentwise weak convergence. Let
$\{\lambda_{n}\}_{n=1}^{\infty}$
 and
$\{\lambda_{n}^{\prime}\}_{n=1}^{\infty}$
 be two sequences in
$\mathbb{T}$. The aim of current section is to study the limiting
behavior of the sequence of pairs
$\{(\mu_{n},\nu_{n})\}_{n=1}^{\infty}$, where
\[
(\mu_{n},\nu_{n})
 =
 (\delta_{\lambda_{n}},\delta_{\lambda_{n}^{\prime}})
 \boxtimes(\mu_{n1},\nu_{n1})\boxtimes(\mu_{n2},\nu_{n2})
 \boxtimes\cdots\boxtimes(\mu_{nk_{n}},\nu_{nk_{n}}),
 \]
 for $\delta_{\lambda}$ the Dirac point mass supported at $\lambda\in\mathbb{T}$.

 We would like to mention at this point that the asymptotic behavior
of boolean convolution $\Utime$ and that of free convolution
$\boxtimes$ has been studied thoroughly in \cite{Boolean}, where the
necessary and sufficient conditions for the weak convergence were
found. These conditions show that the sequence
$\delta_{\lambda_{n}}\utimes\mu_{n1}\utimes\mu_{n2}\utimes\cdots\utimes\mu_{nk_{n}}$
converges weakly if and only if the sequence
$\delta_{\lambda_{n}}\boxtimes\mu_{n1}\boxtimes\mu_{n2}\boxtimes\cdots\boxtimes\mu_{nk_{n}}$
does, provided that the array $\{\mu_{nk}\}_{n,k}$ is infinitesimal.
Moreover, the limit laws are proved to be infinitely divisible (see
\cite{Boolean}), and the boolean and free limits are related in a
quite explicit manner. In the sequel, we will prove the analogous
results for c-free and boolean convolutions.

 To our purposes, we also mention the characterization of infinite
divisibility relative to boolean convolution $\Utime$ and that
relative to free convolution $\boxtimes$. A measure
$\nu\in\mathcal{M}_{\mathbb{T}}$ is \emph{$\Utime$-infinitely
divisible} if, for each $n\in\mathbb{N}$, there exists
$\nu_{n}\in\mathcal{M}_{\mathbb{T}}$ such that
 \[
\nu=\underbrace{\nu_{n}\utimes\nu_{n}\utimes\cdots\utimes\nu_{n}}_{n\:\text{times}}.
 \]
The notion of $\boxtimes$-infinite divisibility for a measure is
defined analogously.

As shown in {[}11], a measure $\nu\in\mathcal{M}_{\mathbb{T}}$ is
$\Utime$-infinitely divisible \emph{} if and only if either $\nu$ is
the Haar measure on $\mathbb{T}$, or the function $B_{\nu}$ can be
expressed as
\[
B_{\nu}(z)=\gamma\exp\left(-\int_{\mathbb{T}}\frac{1+\zeta
z}{1-\zeta z}\, d\sigma(\zeta)\right),\qquad z\in\mathbb{D},
 \]
 where
$\gamma\in\mathbb{T}$ and $\sigma$ is a finite positive Borel
measure on $\mathbb{T}$. In other words, $\nu$ is
$\Utime$-infinitely divisible if and only if either the function
$B_{\nu}(z)=0$ for all $z$ in $\mathbb{D}$, or $0\notin
B_{\nu}(\mathbb{D})$. The notation
 $\nu_{\Utimes}^{\gamma,\sigma}$
will denote the $\Utime$-infinitely divisible measure determined by
$\gamma$ and $\sigma$ via the above formula.

Analogously, a measure
 $\nu\in\mathcal{M}_{\mathbb{T}}^{\times}$ is
$\boxtimes$-infinitely divisible if and only if the function
$\eta_{\nu}^{\langle-1\rangle}$ can be written as
 \[
 \eta_{\nu}^{\langle-1\rangle}(z)
 =
 z\cdot\gamma \exp\left(\int_{\mathbb{T}}\frac{1+\zeta z}{1-\zeta z}\,
d\sigma(\zeta)\right),\qquad z\in\mathbb{D},
 \]
  for some
$\gamma\in\mathbb{T}$ and a finite positive Borel measure $\sigma$
on $\mathbb{T}$. The $\boxtimes$-infinitely divisible measure $\nu$
described above will be denoted by
$\nu_{\boxtimes}^{\gamma,\sigma}$. The Haar measure $m$ is the only
$\boxtimes$-infinitely divisible measure on $\mathbb{T}$ with zero
first moment.

Let us proceed to the proof of the limit theorems for c-free
convolution. We first show that weak convergence of probability
measures can be translated into convergence properties of
corresponding functions $B$ and $\Sigma$.

\begin{prop}
\label{prop5.1} Let
$\{\mu_{n}\}_{n=1}^{\infty}\subset\mathcal{M}_{\mathbb{T}}$ and
$\{\nu_{n}\}_{n=1}^{\infty}\subset\mathcal{M}_{\mathbb{T}}^{\times}$
be two sequences of probability measures, and let $\mu$ be a
probability measure supported on $\mathbb{T}$.
\begin{enumerate}
\item [(i)] The sequence $\mu_{n}$ converges weakly to $\mu$ if and
only if the functions $B_{\mu_{n}}$ converge uniformly on compact
subsets of $\mathbb{D}$ to the function $B_{\mu}$.
\item [(ii)] Suppose that $\{\nu_{n}\}_{n=1}^{\infty}$ converges weakly to
a measure $\nu\in\mathcal{M}_{\mathbb{T}}^{\times}$. Then the
sequence $\mu_{n}$ converges weakly to $\mu$ if and only if there
exists a neighborhood of zero $\mathcal{D}\subset\mathbb{D}$ such
that all functions $\Sigma_{(\mu_{n},\nu_{n})}$ and
$\Sigma_{(\mu,\nu)}$ are defined in $\mathcal{D}$, and the functions
$\Sigma_{(\mu_{n},\nu_{n})}$ converge uniformly on $\mathcal{D}$ to
the function $\Sigma_{(\mu,\nu)}$.
\end{enumerate}
\end{prop}
\begin{proof}
The equivalence in (i) has been observed in \cite{Boolean}, and it
is based on the following identity: \begin{equation}
\frac{1+zB_{\mu}(z)}{1-zB_{\mu}(z)}=\int_{\mathbb{T}}\frac{1+\zeta
z}{1-\zeta z}\, d\mu(\zeta),\qquad
z\in\mathbb{D},\label{eq12}\end{equation}
 which says that the Poisson integral of the measure $d\mu(\,\overline{\zeta}\,)$
is determined by the function $B_{\mu}$.

Let us prove now (ii). Assume that $\{\mu_{n}\}_{n=1}^{\infty}$
converges weakly to $\mu$. Then the hypothesis on weak convergence
of $\{\nu_{n}\}_{n=1}^{\infty}$ implies that there exists a
neighborhood of zero $\mathcal{D}\subset\mathbb{D}$ such that the
functions
 $\eta_{\nu_{n}}^{\langle-1\rangle}$ and
$\eta_{\nu}^{\langle-1\rangle}$ are defined in $\mathcal{D}$
 (hence so are the functions $\Sigma_{(\mu_{n},\nu_{n})}$ and
$\Sigma_{(\mu,\nu)}$),
and the sequence
$\eta_{\nu_{n}}^{\langle-1\rangle}$
converges uniformly on
$\mathcal{D}$ to $\eta_{\nu}^{\langle-1\rangle}$.
  It follows that
there exist $N\in\mathbb{N}$ and a disk
$\mathcal{D}^{\prime}\subset\mathbb{D}$ such that
$\eta_{\nu}^{\langle
-1\rangle}(\mathcal{D})\subset\mathcal{D}^{\prime}$ and
$\eta_{\nu_{n}}^{\langle
-1\rangle}(\mathcal{D})\subset\mathcal{D}^{\prime}$ for every $n>N$.
Also, the Cauchy estimate implies that there exists
$K=K(\mathcal{D}^{\prime})>0$ so that the derivatives
$\left|B_{\mu_{n}}^{\prime}(z)\right|\leq K$ for every
$n\in\mathbb{N}$ and $z\in\mathcal{D}^{\prime}$. Therefore, we have
\begin{eqnarray*}
\left|\Sigma_{(\mu_{n},\nu_{n})}(z)-\Sigma_{(\mu,\nu)}(z)\right|
 & \leq &
 \left|B_{\mu_{n}}
 \left(\eta_{\nu_{n}}^{\langle-1\rangle}(z)\right)
 -B_{\mu_{n}}\left(\eta_{\nu}^{\langle-1\rangle}(z)\right)\right|\\
 &  &
  +\left|B_{\mu_{n}}\left(\eta_{\nu}^{\langle-1\rangle}(z)\right)
  -B_{\mu}\left(\eta_{\nu}^{\langle-1\rangle}(z)\right)\right|\\
 & \leq &
 K\left|\eta_{\nu_{n}}^{\langle-1\rangle}(z)
 -\eta_{\nu}^{\langle-1\rangle}(z)\right|
 +\left|B_{\mu_{n}}\left(\eta_{\nu}^{\langle-1\rangle}(z)\right
 )-B_{\mu}\left(\eta_{\nu}^{\langle-1\rangle}(z)\right)\right|,
 \end{eqnarray*}
for every $n>N$ and $z\in\mathcal{D}$. Then (i) implies that the
functions $\Sigma_{(\mu_{n},\nu_{n})}$ converge uniformly on
$\mathcal{D}$ to the function $\Sigma_{(\mu,\nu)}$.

Conversely, suppose that the functions
  $\Sigma_{(\mu_{n},\nu_{n})}$
converge uniformly on $\mathcal{D}$ to the function
$\Sigma_{(\mu,\nu)}$.
 Observe that
$B_{\mu_{n}}(z)=\Sigma_{(\mu_{n},\nu_{n})}\left(\eta_{\nu_{n}}(z)\right)$
for all $z\in\mathbb{D}$.
A similar argument as in the previous
paragraph shows that there exists a positive constant
$K^{\prime}$
such that the estimate
 \[
 \left|B_{\mu_{n}}(z)-B_{\mu}(z)\right|
  \leq
 K^{\prime}\left|\eta_{\nu_{n}}(z)-\eta_{\nu}(z)\right|
  +\left|\Sigma_{(\mu_{n},\nu_{n})}\left(\eta_{\nu}(z)\right)
  -\Sigma_{(\mu,\nu)}\left(\eta_{\nu}(z)\right)\right|
   \]
holds in a neighborhood of zero, for sufficiently large $n$.
Therefore, the weak convergence of $\{\nu_{n}\}_{n=1}^{\infty}$
implies that the sequence $B_{\mu_{n}}(z)$ converges uniformly on a
neighborhood of zero to the function $B_{\mu}(z)$. Moreover, this
convergence is actually uniform on any compact subset of
$\mathbb{D}$ by an easy application of Montel's theorem. We
therefore conclude, from (i), that the sequence $\mu_{n}$ converges
weakly to $\mu$.
\end{proof}
For an infinitesimal array $\{\mu_{nk}\}_{n,k}$, we define the
complex numbers $b_{nk}\in\mathbb{T}$ by
  \begin{equation}
b_{nk}=\exp\left(i\int_{\left|\arg\zeta\right|<1}\arg\zeta\,
d\mu_{nk}(\zeta)\right)\label{eq13}
   \end{equation}
 where $\arg\zeta$ denotes the principal value of the argument of
$\zeta$, and the probability measures $\mu_{nk}^{\circ}$ by \[
d\mu_{nk}^{\circ}(\zeta)=d\mu_{nk}(b_{nk}\zeta).\]
 The array $\{\mu_{nk}^{\circ}\}_{n,k}$ is infinitesimal and
  $\lim_{n\rightarrow\infty}\max_{1\leq k\leq k_{n}}\left|\arg b_{nk}\right|=0.$
We also associate to each measure $\mu_{nk}^{\circ}$ an analytic
function
 \[
  h_{nk}(z)=-i\int_{\mathbb{T}}\Im\zeta\,
d\mu_{nk}^{\circ}(\zeta)+\int_{\mathbb{T}}\frac{1+\zeta z}{1-\zeta
z}(1-\Re\zeta)\, d\mu_{nk}^{\circ}(\zeta),\qquad z\in\mathbb{D},
 \]
 and observe that
  $\Re h_{nk}(z)>0$ for all $z\in\mathbb{D}$
    unless the measure $\mu_{nk}^{\circ}=\delta_{1}$.

\begin{prop}
\label{prop52} Suppose that $\mathcal{D}\subset\mathbb{D}$ is a disk
centered at zero with radius less than $1/4$. Let
$\{\lambda_{n}\}_{n=1}^{\infty}$ be a sequence in $\mathbb{T}$, and
let $\{\mu_{nk}\}_{n,k}$ be an infinitesimal array in
$\mathcal{M}_{\mathbb{T}}^{\times}.$ Then we have:
\begin{enumerate}
\item [(1)] $1-B_{\mu_{nk}^{\circ}}(z)=h_{nk}\left(\overline{b_{nk}}z\right)(1+o(1))$
uniformly in $k$ and $z\in\mathcal{D}$ as $n\rightarrow\infty$.
\item [(2)]There exists a constant $L=L(\mathcal{D})>0$ such that for
every $n$ and $k$ we have \[ \left|h_{nk}(w)-h_{nk}(z)\right|\leq
L\left|h_{nk}(z)\right|\left|w-z\right|,\qquad z,w\in\mathcal{D}.\]

\item [(3)] The sequence of functions
 $\{\exp(i\arg\lambda_{n}
  +i\sum_{k=1}^{k_{n}}
   \arg b_{nk}-\sum_{k=1}^{k_{n}}h_{nk}(z))\}_{n=1}^{\infty}$
converges uniformly on compact subsets of $\mathbb{D}$ if and only
if the sequence of functions
$\{\lambda_{n}\prod_{k=1}^{k_{n}}B_{\mu_{nk}}(z)\}_{n=1}^{\infty}$
does. Moreover, the two sequences have the same limit function.
\end{enumerate}
\end{prop}
\begin{proof}
(1) and (3) are proved in \cite{Boolean}. To prove (2), let us
consider the analytic function \[
h_{\mu}(z)=\int_{\mathbb{T}}\frac{1+\zeta z}{1-\zeta
z}\,(1-\Re\zeta)d\mu(\zeta),\qquad z\in\mathbb{D},\] for a measure
$\mu\in\mathcal{M}_{\mathbb{T}}$. For $z,w\in\mathcal{D}$,
 we have
\begin{eqnarray*}
\left|h_{\mu}(w)-h_{\mu}(z)\right|
 & \leq &
 \left|w-z\right|\int_{\mathbb{T}}
  \frac{2}{(1-\left|z\right|)(1-\left|w\right|)}\,(1-\Re\zeta)d\mu(\zeta)\\
 & \leq &
 4\left|w-z\right|\int_{\mathbb{T}}(1-\Re\zeta)d\mu(\zeta).
  \end{eqnarray*}
In addition, Harnack's inequality shows that there exists
$M=M(\mathcal{D})>0$ such that
 \[
 \Re\left[\frac{1+\zeta z}{1-\zeta
z}\right]\geq M,\qquad z\in\mathcal{D},\:\zeta\in\mathbb{T}.
 \]
Therefore, we deduce that
 \[
M\int_{\mathbb{T}}(1-\Re\zeta)d\mu(\zeta)\leq\int_{\mathbb{T}}\Re\left[\frac{1+\zeta
z}{1-\zeta z}\right]\,(1-\Re\zeta)d\mu(\zeta)=\Re
h_{\mu}(z)\leq\left|h_{\mu}(z)\right|,
 \] for every
$z\in\mathcal{D}$.
 (2) follows from these considerations.
\end{proof}
\begin{lemma}\label{lemma5.3} For sufficiently large $n$, there
exists a disk $\mathcal{D}\subset\mathbb{D}$ centered at zero such
that \[
1-\overline{b_{nk}}\Sigma_{(\mu_{nk},\nu_{nk})}(z)=h_{nk}(z)(1+u_{nk}(z)),\qquad
z\in\mathcal{D},\,1\le k\leq k_{n},\]
 where the limit \[
\lim_{n\rightarrow\infty}\max_{1\leq k\leq
k_{n}}\left|u_{nk}(z)\right|=0\]
 holds uniformly in $\mathcal{D}$. \end{lemma}

\begin{proof}
Introduce measures \[
d\nu_{nk}^{\circ}(\zeta)=d\nu_{nk}(b_{nk}\zeta),\]
 where the complex numbers $b_{nk}$ are defined as in (10). It was
shown in \cite{MulHincin} that the limits
$\lim_{n\rightarrow\infty}\eta_{\nu_{nk}}(z)=z$ and
$\lim_{n\rightarrow\infty}\eta_{\nu_{nk}^{\circ}}(z)=z$
 hold uniformly in $k$ and on compact subsets of $\mathbb{D}$.
  In particular, it follows that, as $n$ tends to infinity, there exists
a disk $\mathcal{D}\subset\mathbb{D}$ centered at zero such that the
functions $\Sigma_{(\mu_{nk},\nu_{nk})}$
 and
$\Sigma_{(\mu_{nk}^{\circ},\nu_{nk}^{\circ})}$
 are both defined in
$\mathcal{D}$
 and
$\eta_{\nu_{nk}}^{\langle-1\rangle}(z)=z(1+o(1))$
uniformly in $k$ and $z\in\mathcal{D}$.

Then the desired result follows form \ref{prop52}(i),
\ref{prop52}(ii), and the following observation:
 \[
\Sigma_{(\mu_{nk},\nu_{nk})}(z)
 =
 b_{nk}\Sigma_{(\mu_{nk}^{\circ},\nu_{nk}^{\circ})}(z)
  =
  b_{nk}B_{\mu_{nk}^{\circ}}\left(\eta_{\nu_{nk}^{\circ}}^{\langle -1\rangle}(z)\right)
   =
   b_{nk}B_{\mu_{nk}^{\circ}}\left(b_{nk}\eta_{\nu_{nk}}^{\langle -1\rangle}(z)\right).
 \]

\end{proof}
As shown in \cite{Mullimit}, for every neighborhood of zero
$\mathcal{D}\subset\mathbb{D}$ there exists $M=M(\mathcal{D})>0$
such that
 \[
  \left|\Im h_{nk}(z)\right|\leq M\Re h_{nk}(z),\qquad
z\in\mathcal{D},\,1\leq k\leq k_{n},
  \]
 for sufficiently large $n$.

Suppose that $\mathcal{D}\subset\mathbb{D}$ is a neighborhood of
zero. The infinitesimality of the arrays $\{\mu_{nk}\}_{n,k}$ and
$\{\nu_{nk}\}_{n,k}$ implies that the functions
$\Sigma_{(\mu_{nk},\nu_{nk})}(z)$ converge uniformly in $k$ and
$z\in\mathcal{D}$ to $1$ as $n\rightarrow\infty$. It follows that
the principal branch of $\log\Sigma_{(\mu_{nk},\nu_{nk})}(z)$ is
defined in $\mathcal{D}$ for large $n$. Moreover, we have
  \[
\log\overline{b_{nk}}\Sigma_{(\mu_{nk},\nu_{nk})}(z)
 =
 \left[\overline{b_{nk}}\Sigma_{(\mu_{nk},\nu_{nk})}(z)-1\right](1+o(1))
  \]
 uniformly in $k$ and $z\in\mathcal{D}$ when $n$ is sufficiently
large.

To prove Theorem \ref{thm55}, we will need again an auxiliary result
from \cite{Boolean}:

\begin{lemma}\label{lemma5.4} Consider a sequence
  $\{ r_{n}\}_{n=1}^{\infty}\subset\mathbb{R}$
and two triangular arrays $\{ z_{nk}\}_{n,k}$ and
 $\{w_{nk}\}_{n,k}$
of complex numbers. Suppose that
\begin{enumerate}
\item [(1)] $\Re w_{nk}\leq0$ and $\Re z_{nk}\leq0$ for all $n$ and
$k$;
\item [(2)] $z_{nk}=w_{nk}(1+\varepsilon_{nk})$,
 where
  $\lim_{n\rightarrow\infty}\max_{1\leq k\leq k_{n}}\left|\varepsilon_{nk}\right|=0$;
\item [(3)] there exists a constant $M>0$
 such that
  $\left|\Im w_{nk}\right|\leq M\left|\Re w_{nk}\right|$
for all $n,k$.
\end{enumerate}
Then the sequence
$\{\exp(ir_{n}+\sum_{k=1}^{k_{n}}z_{nk})\}_{n=1}^{\infty}$ converges
if and only if the sequence
$\{\exp(ir_{n}+\sum_{k=1}^{k_{n}}w_{nk})\}_{n=1}^{\infty}$ does.
Moreover, the two sequences have the same limit.

\end{lemma}


\begin{thm}
\label{thm55} Let $\{\lambda_{n}\}_{n=1}^{\infty}$ and
$\{\lambda_{n}^{\prime}\}_{n=1}^{\infty}$ be two sequences in
$\mathbb{T}$, and let $\{\mu_{nk}\}_{n,k}$ and $\{\nu_{nk}\}_{n,k}$
be two infinitesimal arrays in $\mathcal{M}_{\mathbb{T}}^{\times}$.
Assume that $\mathcal{D}\subset\mathbb{D}$ is a neighborhood of
zero. Then the sequence of functions
$\{\lambda_{n}\prod_{k=1}^{k_{n}}\Sigma_{(\mu_{nk},\nu_{nk})}(z)\}_{n=1}^{\infty}$
converges uniformly on $\mathcal{D}$ if and only if the sequence of
functions
$\{\lambda_{n}\prod_{k=1}^{k_{n}}B_{\mu_{nk}}(z)\}_{n=1}^{\infty}$
does. Moreover, the two sequences have the same limit function.
\end{thm}
\begin{proof}
For sufficiently large $n$ and $z\in\mathcal{D}$, let us write
 \[
\lambda_{n}\prod_{k=1}^{k_{n}}\Sigma_{(\mu_{nk},\nu_{nk})}(z)
 =\exp\left(i\arg\lambda_{n}+i\sum_{k=1}^{k_{n}}
 \arg b_{nk}+\sum_{k=1}^{k_{n}}\log\overline{b_{nk}}\Sigma_{(\mu_{nk},\nu_{nk})}(z)\right).
 \]
 Applying Lemmas \ref{lemma5.3} and \ref{lemma5.4} to the arrays
$\{\log\overline{b_{nk}}\Sigma_{(\mu_{nk},\nu_{nk})}(z)\}_{n,k}$
and
$\{-h_{nk}(z)\}_{n,k}$, we conclude that the sequences
$\{\exp(i\arg\lambda_{n}+i\sum_{k=1}^{k_{n}}\arg
b_{nk}-\sum_{k=1}^{k_{n}}h_{nk}(z))\}_{n=1}^{\infty}$
 and
$\{\lambda_{n}\prod_{k=1}^{k_{n}}
 \Sigma_{(\mu_{nk},\nu_{nk})}(z)\}_{n=1}^{\infty}$
have the same asymptotic behavior as $n\rightarrow\infty$. The
desired result follows from the fact that the above sequences are
normal families of analytic functions, and from \ref{prop52}(3).
\end{proof}

Fix now $\gamma,\gamma^{\prime}\in\mathbb{T}$ and finite positive
Borel measures $\sigma,\sigma^{\prime}$ on $\mathbb{T}$.
  Recall that
$\nu_{\Utimes}^{\gamma,\sigma}$ (resp.,
$\nu_{\boxtimes}^{\gamma^{\prime},\sigma^{\prime}}$) is the
$\Utime$(resp., $\boxtimes$)-infinitely divisible measure we have
seen earlier.
\begin{thm}
\label{thm5.6} Let $\{\lambda_{n}\}_{n=1}^{\infty}$ and
$\{\lambda_{n}^{\prime}\}_{n=1}^{\infty}$ be two sequences in
$\mathbb{T}$, and let $\{\mu_{nk}\}_{n,k}$ and $\{\nu_{nk}\}_{n,k}$
be two infinitesimal arrays in $\mathcal{M}_{\mathbb{T}}^{\times}$.
Define
 \[
(\mu_{n},\nu_{n})
 =(\delta_{\lambda_{n}},\delta_{\lambda_{n}^{\prime}})
  \boxtimes(\mu_{n1},\nu_{n1})\boxtimes(\mu_{n2},\nu_{n2})
  \boxtimes\cdots\boxtimes(\mu_{nk_{n}},\nu_{nk_{n}})
  \]
and
$\rho_{n}
 =\delta_{\lambda_{n}}\utimes\mu_{n1}\utimes\mu_{n2}\utimes\cdots\utimes\mu_{nk_{n}}$,
for every $n\in\mathbb{N}$.
 Suppose that
$\{\nu_{n}\}_{n=1}^{\infty}$ converges weakly to
$\nu_{\boxtimes}^{\gamma^{\prime},\sigma^{\prime}}$.
Then the following statements are equivalent:
\begin{enumerate}
\item The sequence $\mu_{n}$ converges weakly to a measure $\mu\in\mathcal{M}_{\mathbb{T}}^{\times}$.
\item The sequence $\rho_{n}$ converges weakly to $\nu_{\Utimes}^{\gamma,\sigma}$.
\item The sequence of measures
 \[
d\sigma_{n}(\zeta)=\sum_{k=1}^{k_{n}}(1-\Re\zeta)\,
d\mu_{nk}^{\circ}(\zeta)
 \]
  converges weakly on $\mathbb{T}$ to the
measure $\sigma$, and the sequence of complex numbers
 \[
\gamma_{n}=\exp\left(i\arg\lambda_{n}
 +i\sum_{k=1}^{k_{n}}\arg b_{nk}+i\sum_{k=1}^{k_{n}}
  \int_{\mathbb{T}}\Im\zeta\, d\mu_{nk}^{\circ}(\zeta)\right)
 \]
 converges to $\gamma$ as
$n\rightarrow\infty$.
\end{enumerate}
Moreover, if \textup{(1)}-\textup{(3)} are satisfied, then we have
\[
\Sigma_{\left(\mu,\nu_{\boxtimes}^{\gamma^{\prime},\sigma^{\prime}}\right)}(z)
 =B_{\nu_{\Utimes}^{\gamma,\sigma}}(z)
 \]
in a neighborhood of zero.
\end{thm}
\begin{proof}
The equivalence between (2) and (3) was proved in \cite{Boolean}. We
will show the equivalence of (1) and (2). Note first that we have
\[
\Sigma_{(\mu_{n},\nu_{n})}(z)=\lambda_{n}\prod_{k=1}^{k_{n}}\Sigma_{(\mu_{nk},\nu_{nk})}(z)
\]
 in a neighborhood of zero $\mathcal{D}\subset\mathbb{D}$ and
  \[
B_{\rho_{n}}(z)=\lambda_{n}\prod_{k=1}^{k_{n}}B_{\mu_{nk}}(z),\qquad
z\in\mathbb{D}.\
\]

Suppose that (1) holds. By Proposition \ref{prop5.1}, we have
  \[
\lim_{n\rightarrow\infty}\Sigma_{(\mu_{n},\nu_{n})}(z)
 =
  \Sigma_{\left(\mu,\nu_{\boxtimes}^{\gamma^{\prime},\sigma^{\prime}}\right)}(z)
  \]
 uniformly on $\mathcal{D}$. Then Theorem \ref{thm55} implies that
\[
 \lim_{n\rightarrow\infty}B_{\rho_{n}}(z)
 =
  \Sigma_{\left(\mu,\nu_{\boxtimes}^{\gamma^{\prime},\sigma^{\prime}}\right)}(z)
 \]
 uniformly on $\mathcal{D}$. Note that the function
 $\eta_{\nu_{\boxtimes}^{\gamma^{\prime},\sigma^{\prime}}}^{\langle-1 \rangle}(z)$
can be analytically continued to the whole disk $\mathbb{D}$. It
follows that there exists a unique measure
$\nu\in\mathcal{M}_{\mathbb{T}}$
such that
 \[
B_{\nu}(z)=\Sigma_{\left(
 \mu,\nu_{\boxtimes}^{\gamma^{\prime},\sigma^{\prime}}
  \right)}(z),\qquad z\in\mathcal{D}.
\]
 Note that
 \[
\int_{\mathbb{T}}\zeta\, d\nu(\zeta)=B_{\nu}(0)
 =
  \Sigma_{\left(\mu,\nu_{\boxtimes}^{\gamma^{\prime},\sigma^{\prime}}\right)}(0)
 =
  \int_{\mathbb{T}}\zeta\, d\mu(\zeta)\neq0.
  \]
 We conclude that $\{\rho_{n}\}_{n=1}^{\infty}$
converges weakly to the measure
$\nu\in\mathcal{M}_{\mathbb{T}}^{\times}$ by Proposition 5.1, and
hence the limit law $\nu$ is $\Utime$-infinitely divisible as we
have mentioned earlier. Consequently, the measure $\nu$ is of the
form $\nu_{\Utimes}^{\gamma,\sigma}$ for some $\gamma\in\mathbb{T}$
and a finite Borel measure $\sigma$ on $\mathbb{T}$. Hence (2)
holds.

Assume now (2). Then we have
 \[
\lim_{n\rightarrow\infty}B_{\rho_{n}}(z)=B_{\nu_{\Utimes}^{\gamma,\sigma}}(z)
 \]
 uniformly on compact subsets of $\mathbb{D}$. Theorem \ref{thm55}
shows that
\[
\lim_{n\rightarrow\infty}\Sigma_{(\mu_{n},\nu_{n})}(z)
 =
  B_{\nu_{\Utimes}^{\gamma,\sigma}}(z)
 \]
 uniformly in a neighborhood of zero. Observe now there exists a measure
$\mu\in\mathcal{M}_{\mathbb{T}}^{\times}$ such that
 \[
B_{\mu}(z)=B_{\nu_{\Utimes}^{\gamma,\sigma}}
 \left(
  \eta_{\nu_{\boxtimes}^{\gamma^{\prime},\sigma^{\prime}}}(z)
 \right),
 \qquad z\in\mathbb{D}.
 \]
 Therefore the function $B_{\nu_{\Utimes}^{\gamma,\sigma}}(z)$ has
the form
$\Sigma_{
 \left(
  \mu,\nu_{\boxtimes}^{\gamma^{\prime},\sigma^{\prime}}
   \right)}(z)$
in a neighborhood of zero, and (1) follows from Proposition
\ref{prop5.1}.
\end{proof}
Next, we address the case that the measures $\mu_{n}$ converge to
$m$, for $m$ the Haar measure on $\mathbb{T}$.
\begin{thm}
\label{thm5.7} Let $\{\lambda_{n}\}_{n=1}^{\infty}$,
$\{\lambda_{n}^{\prime}\}_{n=1}^{\infty}$ be sequences in
$\mathbb{T}$, and let $\{\mu_{nk}\}_{n,k}$ , $\{\nu_{nk}\}_{n,k}$ be
two infinitesimal arrays in $\mathcal{M}_{\mathbb{T}}^{\times}$. As
in the statement of \textup{Theorem 5.6}, define
 \[
(\mu_{n},\nu_{n})
 =(\delta_{\lambda_{n}},\delta_{\lambda_{n}^{\prime}})
  \boxtimes(\mu_{n1},\nu_{n1})\boxtimes(\mu_{n2},\nu_{n2})
  \boxtimes\cdots\boxtimes(\mu_{nk_{n}},\nu_{nk_{n}})
  \]
and $\rho_{n}
 =\delta_{\lambda_{n}}\utimes\mu_{n1}\utimes\mu_{n2}\utimes\cdots\utimes\mu_{nk_{n}}$,
for every $n\in\mathbb{N}$,
 and suppose that
$\{\nu_{n}\}_{n=1}^{\infty}$
 converges weakly to $\nu$. Then the
following assertions are equivalent:
\begin{enumerate}
\item [(1)] The sequence $\mu_{n}$ converges weakly to $m$.
\item [(2)] The sequence $\rho_{n}$ converges weakly to $m$.
\end{enumerate}
\end{thm}
\begin{proof}
Assume (1) holds. Observe that
$\Sigma_{(\mu_{n},\nu_{n})}(0)=B_{\mu_{n}}(0)=\int_{\mathbb{T}}\zeta\,
d\mu_{n}(\zeta)$ and
 \[
\Sigma_{(\mu_{n},\nu_{n})}(0)
 =
 \lambda_{n}\prod_{k=1}^{k_{n}}\Sigma_{(\mu_{nk},\nu_{nk})}(0)
  =
  \lambda_{n}\prod_{k=1}^{k_{n}}\int_{\mathbb{T}}\zeta\,
d\mu_{nk}(\zeta),\qquad n\in\mathbb{N}.
 \]
 Since the sequence $\mu_{n}$ converges weakly to $m$ , the above
product converges to zero as $n$ tends to infinity. As shown in
\cite{Boolean}, Theorem 4.3, this condition implies that (2) holds.

Conversely, assume that (2) holds. Then the limit
$\lim_{n\rightarrow\infty}B_{\rho_{n}}(z)=0$
 holds uniformly on the
compact subsets of $\mathbb{D}$, and therefore Theorem \ref{thm55}
shows that
  \[
\lim_{n\rightarrow\infty}\Sigma_{(\mu_{n},\nu_{n})}(z)=0
 \]
 uniformly in a neighborhood of zero
  $\mathcal{D}\subset\mathbb{D}$.
Note first that the measure $\nu$ must be $\boxtimes$-infinitely
divisible (see \cite{MulHincin}). If the measure $\nu$ has nonzero
first moment, then we conclude, by Proposition \ref{prop5.1}, that
(1) holds since the zero function is of the form $\Sigma_{(m,\nu)}$.
If the first moment of the measure $\nu$ is zero, then $\nu=m$ and
the sequence
  $\{\eta_{\nu_{n}}(z)\}_{n=1}^{\infty}$
   converges
uniformly on $\mathcal{D}$ to zero. In particular, the set
$\eta_{\nu_{n}}(\mathcal{D})$ is contained in $\mathcal{D}$ when $n$
is sufficiently large. We conclude in this case that the functions
$B_{\mu_{n}}(z)=\Sigma_{(\mu_{n},\nu_{n})}(\eta_{\nu_{n}}(z))$
converge uniformly on $\mathcal{D}$ to zero as well, and therefore
the sequence $\mu_{n}$ converges weakly to $m$ by Proposition
\ref{prop5.1}.
\end{proof}
\begin{remark}\label{rem5.8} We remark here for a further use that
the proof of Theorem \ref{thm5.7} gives a convenient criterion of
determining whether a limit law is the Haar measure $m$ or not.
Namely, if the free convolutions $\nu_{n}$ converge weakly to $\nu$,
and the c-free convolutions $\mu_{n}$ converge to a measure $\mu$
with $\int_{\mathbb{T}}\zeta\, d\mu(\zeta)=0$, then the measure
$\mu$ must be the Haar measure $m$. \end{remark}


\section{Infinite Divisibility}

\begin{defn}
A pair
$(\mu,\nu)\in\mathcal{M}_{\mathbb{T}}\times\mathcal{M}_{\mathbb{T}}$
is said to be $\boxtimes$-infinitely divisible if for any
$n\in\mathbb{N}$ there exists
$(\mu_{n},\nu_{n})\in\mathcal{M}_{\mathbb{T}}\times\mathcal{M}_{\mathbb{T}}$
such that
\[
 (\mu,\nu)
 =
 \underbrace{(\mu_{n},\nu_{n})
 \boxtimes\cdots\boxtimes(\mu_{n},\nu_{n})}_{n\,\text{times}}.
 \]

\end{defn}
In this section, we give a complete characterization of
$\boxtimes$-infinite divisibility for a pair
$(\mu,\nu)\in\mathcal{M}_{\mathbb{T}}\times\mathcal{M}_{\mathbb{T}}$.
We begin with the case when $\mu$ has zero first moment. We will see
that these are the pairs $(m,\nu)$, where $m$ is the Haar measure on
$\mathbb{T}$ and the measure $\nu$ is $\boxtimes$-infinitely
divisible.

\begin{prop}
\label{prop6.1} Suppose that the pair $(\mu,\nu)$ is
$\boxtimes$-infinitely divisible and that $\int_{\mathbb{T}}\zeta\,
d\mu(\zeta)=0$. Then $\mu=m$.
\end{prop}
\begin{proof}
We first assume that $\int_{\mathbb{T}}\zeta\, d\nu(\zeta)\neq0$. In
this case, for every $n\in\mathbb{N}$, there exist measures
$\mu_{n}\in\mathcal{M}_{\mathbb{T}}$ and
$\nu_{n}\in\mathcal{M}_{\mathbb{T}}^{\times}$ such that
 \[
(\mu,\nu)
=
\underbrace{(\mu_{n},\nu_{n})
\boxtimes\cdots\boxtimes(\mu_{n},\nu_{n})}_{n\,\text{times}}.
 \]
 Thus we have
 $\Sigma_{(\mu,\nu)}=\left(\Sigma_{(\mu_{n},\nu_{n})}\right)^{n}.$
Or, equivalently, we have
 \[
B_{\mu}(z)=\left(B_{\mu_{n}}\left(\eta_{\nu_{n}}^{\langle
-1\rangle}\left(\eta_{\nu}(z)\right)\right)\right)^{n},\qquad
z\in\mathbb{D},
\]
In \cite{Semigroup}, Proposition 3.3, it is shown that the function\
 $\eta_{\nu_{n}}^{\langle -1\rangle}\left(\eta_{\nu}(z)\right)$
 extends analytically to the entire disk $\mathbb{D}$.
 Therefore the function $B_{\mu}$ is the $n$-th power
of an analytic function in $\mathbb{D}$ for any $n\in\mathbb{N}$.
This happens if and only if either $B_{\mu}$ is identically zero in
$\mathbb{D}$ or $0\notin B_{\mu}(\mathbb{D})$. Hence the measure
$\mu$ is the $\Utime$-infinitely divisible measure with first moment
zero, that is, $\mu=m$.

Suppose now $\int_{\mathbb{T}}\zeta\, d\nu(\zeta)=0$. Then we have
that
\[
(\mu,\nu)=(\mu,m)=(\mu_{1},m)\boxtimes(\mu_{1},m)
\]
for some measure $\mu_{1}\in\mathcal{M}_{\mathbb{T}}$. In this case
we can view the pair $(\mu,m)$ as the distribution of $XY$ in a
noncommutative $C^{*}$-probability space
$(\mathcal{A},\varphi,\psi)$, where $X$ and $Y$ are boolean
independent random variables with the common distribution
$(\mu_{1},m)$. Then, for any $n\in\mathbb{N}$, the $n$-th moment of
the measure $\mu$ is given by
 \[
\varphi(\,\underbrace{(XY)\cdots(XY)}_{n\,\text{times}}\,)
 =\varphi(X)^{n}\varphi(Y)^{n}=\varphi(XY)^{n}=\left(\int_{\mathbb{T}}\zeta\,
d\mu(\zeta)\right)^{n}=0.
 \]
 Therefore the measure $\mu$ is the Haar
measure $m$.
\end{proof}
We focus next on the
 $\boxtimes$-infinitely divisible pair $(\mu,m)$
such that $\mu\in\mathcal{M}_{\mathbb{T}}^{\times}$.

\begin{prop}
\label{prop6.2} The above measure $\mu$ is either a point mass
$\delta_{\lambda}$ for some $\lambda\in\mathbb{T}$ or the harmonic
measure for the unit disc relative to some point
$\alpha\in\mathbb{D}\setminus\{0\}$, i.e., $d\mu=P_{\alpha}\, dm$,
where $P_{\alpha}$ is the Poisson kernel at $\alpha$.
\end{prop}

\begin{proof}
By the virtue of (\ref{eq12}), 
 it suffices to show that the function $B_{\mu}$ is a constant function.
Set $c=\int_{\mathbb{T}}\zeta\, d\mu(\zeta)$. Then the same argument
as in the proof of Proposition 6.2 shows that the $n$-th moment of
the measure $\mu$ is $c^{n}$. We conclude that the function
  \[
m_{\mu}(z)=cz+c^{2}z^{2}+\cdots=\frac{cz}{1-cz}
 \]
 in a neighborhood of zero. This fact implies that $B_{\mu}(z)=c$
for all $z\in\mathbb{D}$ as desired.
\end{proof}
It is easy to see that if the measure $\nu$ is
$\boxtimes$-infinitely divisible, a pair of measures of the form
$(m,\nu)$, $(\delta_{\lambda},m)$ or $(P_{\alpha}\, dm,m)$, is also
$\boxtimes$-infinitely divisible.

Finally, we characterize the $\boxtimes$-infinite divisibility in
the class
$\mathcal{M}_{\mathbb{T}}^{\times}\times\mathcal{M}_{\mathbb{T}}^{\times}$.
A family of pairs $\{(\mu_{t},\nu_{t})\}_{t\geq0}$ of probability
measures on $\mathbb{T}$ is a \emph{weakly continuous semigroup}
relative to the convolution $\boxtimes$ if
$(\mu_{t},\nu_{t})\boxtimes(\mu_{s},\nu_{s})=(\mu_{t+s},\nu_{t+s})$
for $t,s\geq0$, and the maps $t\mapsto\mu_{t}$ and $t\mapsto\nu_{t}$
are continuous.

\begin{thm}
\label{thm6.3} Given a measure
$\mu\in\mathcal{M}_{\mathbb{T}}^{\times}$ and a
$\boxtimes$-infinitely divisible measure
$\nu\in\mathcal{M}_{\mathbb{T}}^{\times}$ , the following statements
are equivalent:
\begin{enumerate}
\item The pair $(\mu,\nu)$ is $\boxtimes$-infinitely divisible.
\item There exists a complex number
 $\gamma\in\mathbb{T}$
  and a finite
positive Borel measure $\sigma$ on $\mathbb{T}$ such that
  \[
\Sigma_{(\mu,\nu)}(z)
 =
  \gamma\exp\left(-\int_{\mathbb{T}}\frac{1+\zeta
z}{1-\zeta z}\, d\sigma(\zeta)\right)
  \]
 in a neighborhood of zero.
\item There exists a weakly continuous semigroup
 $\{(\mu_{t},\nu_{t})\}_{t\geq0}$
relative to $\boxtimes$ such that
 $(\mu_{1},\nu_{1})=(\mu,\nu)$
   and
$(\mu_{0},\nu_{0})=(\delta_{1},\delta_{1})$.
\end{enumerate}
Moreover, if statements \textup{(1)} to \textup{(3)} are all
satisfied, then the limit
\[
\gamma=\lim_{t\rightarrow0^{+}}\exp\left(\frac{1}{t}\int_{\mathbb{T}}i\Im\zeta\,
d\mu_{t}(\zeta)\right)
\]
 exists and the measure $\sigma$ is the weak limit of measures
 \[
\frac{1}{t}\,(1-\Re\zeta)\, d\mu_{t}(\zeta)
 \]
 as $t\rightarrow0^{+}$.
\end{thm}
\begin{proof}
We first prove that (1) implies (2). Assume that (1) holds. For
every $n\in\mathbb{N}$, we have
 \[
(\mu,\nu)=\underbrace{(\mu_{n},\nu_{n})
\boxtimes\cdots\boxtimes(\mu_{n},\nu_{n})}_{n\,\text{times}}.
 \]
 where
  $\mu_{n}\in\mathcal{M}_{\mathbb{T}}$ and $\nu_{n}\in\mathcal{M}_{\mathbb{T}}^{\times}$.
The $\boxtimes$-infinite divisibility of the measure $\nu$ implies
that there exists an analytic function $u(z)$ in $\mathbb{D}$ such
that the function
 $\eta_{\nu}^{\langle-1\rangle}(z)=z\exp(u(z))$ and $\Re
u(z)\geq0$ for all $z\in\mathbb{D}$ (see \cite{LevyHincin}). It
follows that $\eta_{\nu_{n}}^{\langle-1\rangle}(z)=z\exp(u(z)/n)$,
and hence we deduce that the measures $\nu_{n}$ converge weakly to
$\delta_{1}$. On the other hand, the identity
$(B_{\mu_{n}}(z))^{n}=\Sigma_{(\mu,\nu)}(\eta_{\nu_{n}}(z))$ and
Proposition 5.1 imply that the measures $\mu_{n}$ converge weakly to
$\delta_{1}$ as well. Define two infinitesimal arrays
$\{\mu_{nk}\}_{n,k}$ and $\{\nu_{nk}\}_{n,k}$ by setting
$\mu_{nk}=\mu_{n}$ and $\nu_{nk}=\nu_{n},$ where $1\leq k\leq n$.
Then the measures $(\mu,\nu)$ can be viewed as the weak limit of
c-free convolutions
$(\mu_{n1},\mu_{n1})\boxtimes\cdots\boxtimes(\mu_{nn},,\nu_{nn}).$
Hence (2) follows from Theorem \ref{thm5.6}.

Now, assume that (2) holds. It was also proved in \cite{LevyHincin}
that there exists a weakly continuous semigroup
$\{\nu_{t}\}_{t\geq0}$ relative to $\boxtimes$ so that
$\nu_{0}=\delta_{1}$ and $\nu_{1}=\nu$.
 Note that, for every
$t\geq0$, there exists a unique probability measure $\mu_{t}$ on
$\mathbb{T}$ such that
$B_{\mu_{t}}(z)=\left(\Sigma_{(\mu,\nu)}(\eta_{\nu_{t}}(z))\right)^{t}$
for all $z\in\mathbb{D}$, where $\mu_{0}=\delta_{1}$. Then it is
easy to see that the convolution semigroup
$\{(\mu_{t},\nu_{t})\}_{t\geq0}$ has the desired properties.

The implication from (3) to (1) is obvious. To conclude, we only
need to show the assertions about the measure $\sigma$ and the
number $\gamma$. Assume that the pair $(\mu,\nu)$ is
$\boxtimes$-infinitely divisible, and let
$\{(\mu_{t},\nu_{t})\}_{t\geq0}$ be the corresponding convolution
semigroup as in (3). Let $\{ t_{n}\}_{n=1}^{\infty}$ be a sequence
of positive real numbers such that
$\lim_{n\rightarrow\infty}t_{n}=0$. Let $k_{n}=[1/t_{n}]$ for every
$n\in\mathbb{N}$, where $[x]$ denotes the largest integer that is no
greater than the real number $x$. Observe that \[
1-t_{n}<t_{n}k_{n}\leq1,\qquad n\in\mathbb{N}.\]
 Hence we have $\lim_{n\rightarrow\infty}t_{n}k_{n}=1$, and further
the properties of the semigroup $\{(\mu_{t},\nu_{t})\}_{t\geq0}$
show that the c-free convolutions
 \[
\underbrace{(\mu_{t_{n}},\nu_{t_{n}})
 \boxtimes(\mu_{t_{n}},\nu_{t_{n}})
 \boxtimes\cdots
  \boxtimes(\mu_{t_{n}},\nu_{t_{n}})}_{k_{n}\,\text{times}}
   =(\mu_{t_{n}k_{n}},\nu_{t_{n}k_{n}})
    \]
 converge weakly to
  $(\mu_{1},\nu_{1})=(\mu,\nu)$ as $n\rightarrow\infty$.
Then Theorem \ref{thm5.6} implies that the measures
 \[
\frac{1}{t_{n}}\,(1-\Re\zeta)\,
d\mu_{t_{n}}^{\circ}(\zeta)=\frac{1}{t_{n}k_{n}}\,
k_{n}(1-\Re\zeta)\, d\mu_{t_{n}}^{\circ}(\zeta)
 \]
 converge weakly to the measure $\sigma$ and
  \[
\gamma
 =
 \lim_{n\rightarrow\infty}\exp\left(\frac{1}{t_{n}}\int_{\mathbb{T}}i\Im\zeta\,
d\mu_{t_{n}}^{\circ}(\zeta)\right),
 \]
 where the centered measures $d\mu_{t_{n}}^{\circ}(\zeta)=d\mu_{t_{n}}(b_{n}\zeta)$
and the numbers $b_{n}$ are defined as in (10). The desired result
follows immediately from the fact that
$\lim_{n\rightarrow\infty}b_{n}=1$, and that the topology on the set
$\mathcal{M}_{\mathbb{T}}$ determined by the weak convergence of
measures is actually metrizable.
\end{proof}
We conclude this section by showing a c-free analogue of
Hin\v{c}in's classical theorem on the infinite divisibility of real
random variables \cite{Hincin}.

\begin{cor}
\label{cor6.4} Let $\{\lambda_{n}\}_{n=1}^{\infty}$ and
$\{\lambda_{n}^{\prime}\}_{n=1}^{\infty}$ be two sequences in
$\mathbb{T}$, and let $\{\mu_{nk}\}_{n,k}$ and $\{\nu_{nk}\}_{n,k}$
be two infinitesimal arrays in $\mathcal{M}_{\mathbb{T}}^{\times}$.
Suppose that the sequence
 \[
  \left\{
(\delta_{\lambda_{n}},\delta_{\lambda_{n}^{\prime}})
 \boxtimes(\mu_{n1},\nu_{n1})\boxtimes(\mu_{n2},\nu_{n2})
 \boxtimes\cdots\boxtimes(\mu_{nk_{n}},\nu_{nk_{n})}
 \right\}_{n}
 \]
 converges weakly to $(\mu,\nu).$ Then the pair $(\mu,\nu)$ is $\boxtimes$-infinitely
divisible.
\end{cor}
\begin{proof}
Note first that the measure $\nu$ is $\boxtimes$-infinitely
divisible as we have seen earlier. The case
$\mu,\nu\in\mathcal{M}_{\mathbb{T}}^{\times}$ is an application of
Theorems \ref{thm5.6} and \ref{thm6.3}. Remark \ref{rem5.8} shows
that $\mu=m$ when $\int_{\mathbb{T}}\zeta\, d\mu(\zeta)=0$. Only the
case $\mu\in\mathcal{M}_{\mathbb{T}}^{\times}$ and $\nu=m$ requires
a proof. In this case we set
 \[
(\mu_{n},\nu_{n})
 =
  (\delta_{\lambda_{n}},\delta_{\lambda_{n}^{\prime}})
   \boxtimes(\mu_{n1},\nu_{n1})\boxtimes(\mu_{n2},\nu_{n2})
   \boxtimes\cdots\boxtimes(\mu_{nk_{n}},\nu_{nk_{n}})
   \]
 Observe that $B_{\mu_{n}}(z)=\Sigma_{(\mu_{n},\nu_{n})}(\eta_{\nu_{n}}(z))$,
and that the family
$\{\Sigma_{(\mu_{n},\nu_{n})}(z)\}_{n=1}^{\infty}$ is uniformly
Lipschitz in a neighborhood of zero $\mathcal{D}$ with a Lipschitz
constant $K>0$. We have
 \begin{eqnarray*}
\left|B_{\mu_{n}}(z)-\int_{\mathbb{T}}\zeta\, d\mu(\zeta)\right|
  & \leq &
  \left|B_{\mu_{n}}(z)-B_{\mu_{n}}(0)\right|+\left|B_{\mu_{n}}(0)-B_{\mu}(0)\right|\\
 & \leq &
  K\left|\eta_{\nu_{n}}(z)\right|+\left|B_{\mu_{n}}(0)-B_{\mu}(0)\right|.
  \end{eqnarray*}
 It follows that the functions $B_{\mu_{n}}$ converge uniformly on
compact subsets of $\mathbb{D}$ to the constant function
$\int_{\mathbb{T}}\zeta\, d\mu(\zeta)$, and hence the function
$B_{\mu}$ is the same constant function. The measure $\mu$ in this
case is either a point mass concentrated at a point on $\mathbb{T}$
or the harmonic measure for $\mathbb{D}$ relative to a point in
$\mathbb{D}\setminus\{0\}$. Therefore $(\mu,m)$ is
$\boxtimes$-infinitely divisible.
\end{proof}


\section{Appendix:Non-crossing linked partitions
and a formula for the coefficients of the $T$- and $\cT$-transforms}

\subsection{Non-crossing linked partitions}

The notion of non-crossing linked partitions, that we will largely
use in Section 4, is a generalization of non-crossing partitions. It
was first discussed in \cite{dykema}, in connection to the
$T$-transform.

 By a non-crossing linked partition $\gamma$ of the ordered set
 $\{1,2,\dots,n\}$ we will understand a collection $B_1,\dots, B_k$
 of subsets of $\{1,2,\dots,n\}$, called blocks, with the following
 properties:
 \begin{enumerate}
 \item[(1)]$\displaystyle{\bigcup_{l=1}^kB_l=\{1,\dots,n\}}$
 \item[(2)]$B_1,\dots,B_k$ are non-crossing, in the sense that there
 are no two blocks $B_l, B_s$ and $i<k<p<q$ such that $i,p\in B_l$
 and $k,q\in B_s$.
 \item[(3)] for any $1\leq l,s\leq k$, the intersection $B_l\bigcap
 B_s$ is either void or contains only one element. If
 $\{j\}=B_i\bigcap B_s$, then $|B_s|, |B_l|\geq 2$ and $j$ is the
 minimal element of only one of the blocks $B_l$ and $B_s$.
\end{enumerate}

An element will be said to be \emph{singly}, respectively
\emph{doubly} covered by $\gamma$ if it is contained in exactly one,
respectively exactly two blocks of $\gamma$. The set of all
non-crossing linked partitions on $\{1,\dots, n\}$ will be denoted
by $NCL(n)$. If $\gamma\in NCL(n)$ and $B=i_1<i_2<\dots <i_p$ is a
block of $\gamma$, the element $i_1$ will be denoted $min(B)$. The
block $B$ will be called \emph{exterior} if there is no other block
$D$ of $\gamma$ containing two elements $l,s$ such that $l\leq
i_1<i_p<s$. The set of all exterior blocks of $\gamma$ will be
denoted by $ext(\gamma)$; the set of all  blocks of $\gamma$ which
are not exterior will be denoted by $int(\gamma)$. We will use the
notation $NCL_1(n)$ for the set of all elements in $NCL(n)$ with
only one exterior block.

Example: Below is represented graphically the non-crossing linked
partition\\
 $\gamma=(1,4,6,9), (2,3), (4,5),(6,7,8),(10,11),
(11,12)$ $\in NCL(12)$:

 \setlength{\unitlength}{.17cm}
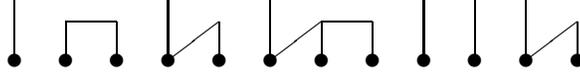
 \begin{equation*}
 \begin{picture}(10,8)

\put(-16,3){\circle*{1}} \put(-16,3){\line(0,1){5}}

\put(-12,3){\circle*{1}} \put(-12,3){\line(0,1){3}}

\put(-8,3){\circle*{1}}\put(-8,3){\line(0,1){3}}

\put(-4,3){\circle*{1}}\put(-4,3){\line(0,1){5}}

\put(0,3){\circle*{1}}\put(0,3){\line(0,1){3}}

\put(4,3){\circle*{1}}\put(4,3){\line(0,1){5}}

\put(8,3){\circle*{1}}\put(8,3){\line(0,1){3}}

\put(12,3){\circle*{1}}\put(12,3){\line(0,1){3}}

\put(16,3){\circle*{1}}\put(16,3){\line(0,1){5}}

\put(20,3){\circle*{1}}\put(20,3){\line(0,1){5}}

\put(24,3){\circle*{1}}\put(24,3){\line(0,1){5}}

\put(28,3){\circle*{1}}\put(28,3){\line(0,1){3}}

\put(-16,8){\line(1,0){32}}

\put(20,8){\line(1,0){4}}

\put(-12,6){\line(1,0){4}}

\put(8,6){\line(1,0){4}}

\put(-4,3){\line(4,3){4}}

\put(4,3){\line(4,3){4}}

\put(24,3){\line(4,3){4}}

 \end{picture}
 \end{equation*}
One has that $ext(\gamma)=(1,4,6,9),(10,11)$ and $int(\gamma)=(2,3),
(4,5),(6,7,8),(11,12)$.

Note that $NC(n)\subset NCL(n)$ and, for a partition in $NC(n)$,
hence also in $NCL(n)$, blocks that are exterior, respectively
interior, in $NC(n)$ are also in $NCL(n)$.

 If $\pi\in NCL(p)$ and $\sigma\in NCL(s)$, then $\pi\oplus\sigma$ is
 the partition from $NCL(p+s)$ obtained by the juxtaposition of
 $\pi$ and $\sigma$.

 Finally, if $\gamma\in NCL(n)$ and $A$ is a (ordered) subset of $\{1,2,\dots,
 n\}$ with $m$ elements, then by $\gamma_{|A}$ we will understand
 the partition in $NCL(m)$ obtained by intersecting the blocks of
 $\gamma$ with $A$, then shifting the elements of the ordered set
 $A$ into $\{1,2,\dots, m\}$.

 \subsection{A formula for the coefficients of the $T$- and
 $\cT$-transforms}${}$\\

 Let $X\in\gA$ and $\displaystyle{m(z)=\sum_{n=1}^\infty m_nz^n}$,
 $\displaystyle{M(z)=\sum_{n=1}^\infty M_nz^n}$ be the moment generating series of $X$ with
 respect to $\psi$, respectively $\varphi$, while
 $\displaystyle{T(z)=\sum_{n=1}^{\infty}t_n z^n}$ and  $\displaystyle{\cT(z)=\sum_{n=1}^{\infty}{}^ct_n
 z^n}$ be the $T$-, respectively the
 $\cT$-transform of $X$.

 \begin{lemma}\label{nclprop}With the above notations, one has that:
\begin{enumerate}
\item[(i)]$\displaystyle{\left[T\circ\left(m(z)\right)\right](1+m(z))=\frac{m(z)}{z}}$
\item[(ii)]$\displaystyle{\left[\cT\circ\left(m(z)\right)\right](1+M(z))=\frac{M(z)}{z}}$
\end{enumerate}
\end{lemma}
\begin{proof}
(i): By definition,
 $\displaystyle{
 T(z)=\left[\frac{1}{z}R(z)\right]\circ\left[R^{\langle-1\rangle}(z)\right].
 }$

 Since $m(z)=R(z[1+m(z)])$, composing at left with
 $R^{\langle-1\rangle}$and at right with $m^{\langle-1\rangle}$, we
 get
  \[
  R^{\langle-1\rangle}(z)=m^{\langle-1\rangle}(z)(1+z),
  \]
  therefore
  \[
  T(z)=\left[\frac{1}{z}R(z)\right]\circ\left[(1+z)m^{\langle-1\rangle}(z)
  \right].
  \]

  It suffices then to show that
  \[
\left[\frac{1}{z}R(z)\right]\circ\left[(1+z)m^{\langle-1\rangle}(z)
  \right]\circ m(z)=\frac{m(z)}{z[1+m(z)]}.
  \]
  But, from (\ref{defRx}),
  \[
  \frac{m(z)}{z[1+m(z)]}=\left[\frac{1}{z}R(z)\right]\circ\left(z[1+m(z)]\right)
  \]
  and, since,
  \[
  \left[(1+z)m^{\langle-1\rangle}(z)\right]\circ m(z)=z[1+m(z)]
  \]
  we have q.e.d..

  (ii): Analogously, it suffices to show that
  \[
  \left[\frac{1}{z}\cR(z)\right]\circ\left[(1+z)m^{\langle-1\rangle}(z)
  \right]\circ m(z)=\frac{M(z)}{z[1+M(z)]}.
  \]

  (\ref{defcRx}) implies
  \[
  \frac{M(z)}{z[1+M(z)]}=\left(\frac{1}{z}\cR(z)\right)\circ\left(z[1+m(z)]\right)
  \]
  and, since,
  \[
  \left[(1+z)m^{\langle-1\rangle}(z)\right]\circ m(z)=z[1+m(z)]
  \]
  we have again q.e.d..
\end{proof}

\begin{conseq}\label{nclcconseq}
 Let us denote, to ease the writing, $m_0=M_0=1$. Then:
\begin{enumerate}
\item[(i)] $\displaystyle{
 m_n=\sum_{k=1}^{n-1}
 \sum_{\substack{p_1+\cdots+p_k\leq n-1\\ p_1,\dots,p_k\geq 0}}
 \left(t_km_{p_1}\cdots m_{p_k}\right)m_{n-1-(p_1+\dots+p_k)}
  }$
\item[(ii)] $\displaystyle{
 M_n=\sum_{k=1}^{n-1}
 \sum_{\substack{p_1+\cdots+p_k\leq n-1\\ p_1,\dots,p_k\geq 0}}
 \left(t_km_{p_1}\cdots M_{p_k}\right)m_{n-1-(p_1+\dots+p_k)}
  }$
  \end{enumerate}
\end{conseq}

\begin{thm}\label{nclthm} With the above notations, one has that:
\begin{enumerate}
\item[(i)]$\displaystyle{m_n=\sum_{\gamma\in NCL(n)}t_0^{n-|\gamma|}
 \cdot\prod_{B\in\gamma}t_{|B|-1}}$
\item[(ii)]$\displaystyle{M_n=\sum_{\gamma\in NCL(n)}t_0^{n-|\gamma|}
 \cdot\prod_{B\in ext(\gamma)}{}^ct_{|B|-1}
 \cdot\prod_{B\in int(\gamma)}t_{|B|-1}}.$
\end{enumerate}
\end{thm}
\begin{proof}
Part (i) of the result is also shown in \cite{dykema}. The proof
presented below is shorter and employs a different approach.

 First, for $\gamma\in NCL(n)$, let us denote
 \begin{eqnarray*}
 \mathcal{E}(\gamma)
 &=&
 t_0^{n-|\gamma|}\prod_{B\in\gamma}t_{|B|-1}\\
 {}^c\mathcal{E}(\gamma)
 &=&
 t_0^{n-|\gamma|}\prod_{B\in ext(\gamma)}{}^ct_{|B|-1}
 \prod_{B\in int(\gamma)}t_{|B|-1}
 \end{eqnarray*}
 Note that if $\gamma=\gamma_1\oplus\gamma_2$, then
 $\mathcal{E}(\gamma)=\mathcal{E}(\gamma_1)\mathcal{E}(\gamma_2)$
 and
 ${}^c\mathcal{E}(\gamma)={}^c\mathcal{E}(\gamma_1){}^c\mathcal{E}(\gamma_2)$.

 Fix $\gamma\in NCL(n)$ and let $F=1<i_2<\dots<i_k$ be the first
 block of $\gamma$. Particularly, $F\in ext(\gamma)$. Let
 $\overline{F}$ be the smallest subset of $\{1,2,\dots,n\}$ with the
 following properties:
 \begin{enumerate}
 \item[(i)]$F\subset \overline{F}$
 \item[(ii)]if $j\in \overline{F}$, then $\{1,2,\dots,
 j\}\subset\overline{F}$.
 \item[(iii)] if $j\in F$ and $j\in D\in\gamma$, then $D\in
 \overline{F}$.
 \end{enumerate}

 Since
  $\gamma=\gamma_{|\overline{F}}\oplus\gamma_{|\{1,\dots,n\}\setminus \overline{F}}$
   , it follows that
   \[
    NCL(n)=\bigsqcup_{k=1}^{n} NCL_1(k)\oplus NCL(n-k),
   \]
   hence
   \[
   \sum_{\gamma\in NCL(n)} \mathcal{E}(\gamma)
    =\sum_{k=1}^n\sum_{\sigma\in NCL_1(k)}\mathcal{E}(\sigma)\sum_{\pi\in
    NCL(n-k)}\mathcal{E}(\pi).
   \]
   Applying the induction hypothesis, the above equality becomes
   \begin{equation}\label{ncleq1}
 \sum_{\gamma\in NCL(n)} \mathcal{E}(\gamma)
    =\sum_{k=1}^n\sum_{\sigma\in NCL_1(k)}\mathcal{E}(\sigma)m_{n-k}.
   \end{equation}

   Say now $\sigma\in NCL_1(p)$ and $F=1<i_2<\dots<i_k$ is the
   exterior block of $\sigma$. For $2\leq l\leq k-1$ define
   \[
   \sigma(l)=\left\{
   \begin{array}{clcr}
   \sigma_{|\{i_l,i_l+1,\dots,i_{l+1}-1\}}
   &
   \text{if $i_l$ is singly covered}\\
   \sigma_{|\{i_l,i_l+1,\dots,i_{l+1}-1\}}\setminus (i_l)
   &
   \text{if $i_l$ is doubly covered}\\
   \end{array}
   \right.
   \]
   also let
   \[
\sigma^\prime(k)=\left\{
   \begin{array}{clcr}
   \sigma_{|\{i_k,i_k+1,\dots,p\}}
   &
   \text{if $i_k$ is singly covered}\\
   \sigma_{|\{i_k,i_k+1,\dots,p\}}\setminus (i_k)
   &
   \text{if $i_k$ is doubly covered}\\
   \end{array}
   \right.
   \]
   and
   \[
   \sigma(k)=\sigma^\prime(k)
   \oplus\sigma_{|\{2,3,\dots,i_2-1\}}
   \]

   Example: The partition $\sigma=(1,4,5,9), (2,3),
   (5,6,7),(8),(9,10)$ from $NCL_1(10)$ is represented graphically
   below:

    \setlength{\unitlength}{.16cm}
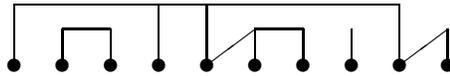
 \begin{equation*}
 \begin{picture}(10,8)

\put(-16,3){\circle*{1}} \put(-16,3){\line(0,1){5}}

\put(-12,3){\circle*{1}} \put(-12,3){\line(0,1){3}}

\put(-8,3){\circle*{1}}\put(-8,3){\line(0,1){3}}

\put(-4,3){\circle*{1}}\put(-4,3){\line(0,1){5}}

\put(0,3){\circle*{1}}\put(0,3){\line(0,1){5}}

\put(4,3){\circle*{1}}\put(4,3){\line(0,1){3}}

\put(8,3){\circle*{1}}\put(8,3){\line(0,1){3}}

\put(12,3){\circle*{1}}\put(12,3){\line(0,1){3}}

\put(16,3){\circle*{1}}\put(16,3){\line(0,1){5}}

\put(20,3){\circle*{1}}\put(20,3){\line(0,1){3}}

\put(-16,8){\line(1,0){32}}

\put(-12,6){\line(1,0){4}}

\put(8,6){\line(-1,0){4}}

\put(0,3){\line(4,3){4}}

\put(16,3){\line(4,3){4}}

 \end{picture}
 \end{equation*}
 and we have that
 \begin{eqnarray*}
 \sigma(2)&=&(1)\\
 \sigma(3)&=&(1,2,3)(4)\\
 \sigma(4)&=&(1,2)(3,4)
 \end{eqnarray*}

 Note that, since $\sigma\in NCL_1(p)$, we have that
 $\sigma^\prime(k)\in NCL(p-k+1)$. Furthermore, the knowledge of
 $\sigma(2),\dots,\sigma(k)$ determines $\sigma$ uniquely and
 \[
 {}^c\mathcal{E}(\sigma)={}^ct_{k-1}\prod_{l=2}^k\mathcal{E}\left(\sigma(l)\right).
 \]
Combining the above result with (\ref{ncleq1}) and the induction
hypothesis, we get
\begin{eqnarray*}
\sum_{\gamma\in NCL(n)} {}^c\mathcal{E}(\gamma)
    &=&
    \sum_{k=1}^n M_{n-k}
    \sum_{\sigma\in
    NCL_1(k)}{}^c\mathcal{E}(\sigma)\\
    &=&
    \sum_{k=1}^{n-1}M_{n-1-k}
    \sum_{
    \substack{p_1+\dots+p_k\leq n-1\\
    p_1,\dots,p_k\geq 1}
    }
    \sum_{
    \substack{\gamma_l\in NCL(p_l)\\
    1\leq l\leq k}
    }
    {}^ct_k
    \prod _{l=1}^k\mathcal(\gamma_k)\\
    &=&
    \sum_{k=1}^{n-1}M_{n-1-k}
    \sum_{
    \substack{p_1+\dots+p_k\leq n-1\\
    p_1,\dots,p_k\geq 1}
    }
    {}^ct_k
    \prod _{l=1}^k
    \left(
    \sum_{\gamma\in NCL(p_k)}
    \mathcal{E}(\gamma)
    \right)\\
    &=&
    \sum_{k=1}^{n-1}M_{n-1-k}
    \sum_{
    \substack{p_1+\dots+p_k\leq n-1\\
    p_1,\dots,p_k\geq 1}
    }
    {}^ct_k
    m_{p_1}\cdots m_{p_k}\\
    &=&
    M_n
\end{eqnarray*}

The inductive step for $m_n$ is analogous (for $\varphi=\psi$, the
sequences $\{m_n\}_n$ and $\{M_n\}_n$ coincide).
\end{proof}

\textbf{Acknowledgements.} We thank \c{S}erban Belinschi and Hari
Bercovici for their constant support and his many advices during the
work on this paper.

\end{document}